\documentclass[11pt]{amsart}

\allowdisplaybreaks[1]         

\def\na{\nabla}

\newcommand{\p}{\partial}

\usepackage{amsmath}
\usepackage{amsfonts}
\usepackage{amssymb}
\usepackage{amsthm}
\usepackage{mathtools} 
\usepackage{latexsym}

\usepackage{enumerate}
\usepackage{cancel}
\usepackage{cases}

  \usepackage{caption}
  \usepackage{subcaption}
  \usepackage{graphics} 
  \usepackage{epsfig} 
\usepackage{graphicx} 

\usepackage{pgfplots}
\pgfplotsset{compat=1.13}

\usepackage{float}

\usepackage{mathrsfs}
\usepackage{fontenc} 
\usepackage{inputenc} 

\usepackage{verbatim}

\theoremstyle{plain}
\newtheorem{theorem}{Theorem}[section]
\newtheorem{proposition}[theorem]{Proposition}
\newtheorem{lemma}[theorem]{Lemma}
\newtheorem{corollary}[theorem]{Corollary}

\theoremstyle{definition}
\newtheorem{definition}[theorem]{Definition}

\theoremstyle{remark}
\newtheorem{remark}[theorem]{Remark}

\numberwithin{equation}{section} 
\numberwithin{figure}{section}   

\usepackage[square,comma,numbers,sort&compress]{natbib}
\usepackage[colorlinks=true, pdfborder={ 0 0 0}]{hyperref}
\hypersetup{urlcolor=blue, citecolor=red}
\usepackage{url}

\newcommand{\vect}[1]{\mathbf{#1}}

\newcommand{\bk}{\vect{k}}

\newcommand{\bu}{\vect{u}}
\newcommand{\bv}{\vect{v}}
\newcommand{\bw}{\vect{w}}

\newcommand{\bx}{\vect{x}}

\newcommand{\bz}{\vect{z}}

\newcommand{\field}[1]{\mathbb{#1}}
\newcommand{\nN}{\field{N}}
\newcommand{\nZ}{\field{Z}}

\newcommand{\nR}{\field{R}}

\newcommand{\bphi}{\boldsymbol{\phi}}
\newcommand{\bomega}{\boldsymbol{\omega}}
\newcommand{\boldeta}{\boldsymbol{\eta}}

\newcommand{\nT}{\mathbb T}

\newcommand{\lp}{\left(}
\newcommand{\rp}{\right)}

\newcommand{\la}{\left\langle}
\newcommand{\ra}{\right\rangle}

\newcommand{\dt}{\partial_t}
\newcommand{\lap}{\triangle}
\newcommand{\grad}{\nabla}
\newcommand{\bee}{\boldsymbol{\eta}^\epsilon}
\newcommand{\R}{\mathbb{R}}
\newcommand{\T}{\mathbb{T}^2}

\newcommand{\cnj}[1]{\overline{#1}}

\newcommand{\abs}[1]{\left\lvert#1\right\rvert}
\newcommand{\set}[1]{\left\{#1\right\}}

\newcommand{\pnt}[1]{\left(#1\right)}

\newcommand{\norm}[1]{\left\|#1\right\|}

\newcommand{\normLp}[2]{\left\|#2 \right\|_{L^{#1}}}
\newcommand{\normHs}[2]{\left\|#2 \right\|_{H^{#1}}}

\usepackage{url}

\usepackage{placeins} 

\newcounter{my_counter}
\usepackage{enumitem}

\title[2D Kuramoto-Sivashinsky: Algebraic Calming]
{Algebraic calming for the 2D Kuramoto-Sivashinsky equations}

\date{\today}

\author{Matthew Enlow}
\address[Matthew Enlow]{Department of Mathematics, 
                University of Nebraska--Lincoln,
        Lincoln, NE 68588-0130, USA}
\email[Matthew Enlow]{menlow2@huskers.unl.edu}

\author{Adam Larios}
\address[Adam Larios]{Department of Mathematics, 
                University of Nebraska--Lincoln,
        Lincoln, NE 68588-0130, USA}
\email[Adam Larios]{alarios@unl.edu}

\author{Jiahong Wu}
\address[Jiahong Wu]{Department of Mathematics, 
                Oklahoma State University
        Stillwater, OK 74078, USA}
\email[Jiahong Wu]{jiahong.wu@okstate.edu}


\begin{document}
\begin{abstract}
We propose an approximate model for the 2D Kuramoto-Sivashinsky equations (KSE) of flame fronts and crystal growth.  We prove that this new ``calmed'' version of the KSE is globally well-posed, and moreover, its solutions converge to solutions of the KSE on the time interval of existence and uniqueness of the KSE at an algebraic rate.  In addition, we provide simulations of the calmed KSE, illuminating its dynamics.  These simulations also indicate that our analytical predictions of the convergence rates are sharp.  We also discuss analogies with the 3D Navier-Stokes equations of fluid dynamics.
\end{abstract}

\maketitle
\thispagestyle{empty}
\section{Introduction}\label{secInt}
\noindent
The Kuramoto-Sivashinsky equation (KSE) is a captivating model for flame fronts, crystal growth, and many other phenomena.  It is both satisfying and frustrating.  In one space dimension, the model acts as a fantastic toy model: it has highly non-trivial chaotic dynamics while still being amenable to a wide range of analytical tools.  However, in higher dimensions, it has so far resisted nearly every analytical attack due to its lack of any known conserved quantity, and the basic question of global well-posedness of solutions remains open, even in two dimensions.  Moreover, the nonlinearity of the system has many similarities with the nonlinearity of the Navier-Stokes equations (NSE), making investigation of the KSE even more intriguing.  

How does one proceed in the face of such difficulty?  In the case of the NSE, at least one approach has been fruitful since at least the work of Smagorinsky in 1963 \cite{Smagorinsky1963}, where a modification of the Navier--Stokes system was proposed, resulting in a system which is both globally well-posed \cite{Ladyzhenskaya1968modifications}, and less computationally demanding to simulate.  Since then, hundreds of so-called ``turbulence models'' have arisen (see, e.g.,  \cite{Rodi2017turbulence,DuraisamyIaccarinoXiao2019turbulence} for a survey),  which typically modify the equations in some way.  It is therefore natural to ask whether such an approach might work for the 2D KSE.\footnote{Since the KSE governs the evolution of a surface, its natural space dimension is two.  Moreover, it is not clear that the 3D case for the KSE is fundamentally more difficult than the 2D case, due to the already strong dissipation.  Hence, we focus on the 2D case.}  However, one quickly realizes that approaches which work for the NSE are unlikely to work for the KSE.  Indeed, for the NSE, the problem is the growth of large gradients; more specifically, the problem is the development of large vorticity, $\bomega:=\nabla\times\bu$ (see, e.g., \cite{Beale_Kato_Majda_1984}).  This is due to the cubic nature of the vorticity equation: $\frac{d}{dt}\|\bomega\|_{L^2}^2\sim (\bomega\cdot\nabla\bu,\bomega)$.  Hence, in order to handle the NSE, one typically attempts to control the gradient of the solution, for example, by strengthening the viscosity or weakening the nonlinear term, since the nonlinear term cascades energy from large scales to small scales, intensifying the gradient.  That is, the NSE are appeased by controlling the \textit{small scales}.  On the other hand, for the KSE, the problem is the growth of large scales. This is due to the cubic nature of the energy equation: $\frac{d}{dt}\|\bu\|_{L^2}^2\sim (\bu\cdot\nabla\bu,\bu)$.  In the 1D case (and in the NSE case), this latter term vanishes, but not in the 2D KSE case.  Moreover, controlling the small scales is not a major problem, as the KSE has a fourth-order diffusion term, strongly curbing the growth of gradients.  Therefore, the problem for the KSE appears to be the exact opposite of the problem for the NSE.  That is, the KSE is appeased by controlling the \textit{large scales}.  Hence, the standard approaches that work for the NSE are unlikely to be of use for the KSE (see \cite{Kostianko_Titi_Zelik_2018} for investigations of this notion in the 1D case), and new approaches for handling the KSE are required.  The purpose of the present work is to propose and investigate one such approach.

In \cite{Larios_Yamazaki_2020_rKSE} the authors study a modification of the 2D KSE that they call the ``reduced KSE'' (r-KSE) with an adjustment made to the linear term in one component. This system admits a maximum principle, allowing for a proof of globally well-posedness.  Moreover, simulations in \cite{Larios_Yamazaki_2020_rKSE} indicate that the dynamics of the r-KSE are arguably qualitatively similar to KSE.  However, r-KSE suffers from the drawback that there is no clear way to see solutions of the r-KSE converge to solutions of the KSE, as any introduction of a ``turning'' parameter interpolating between the r-KSE and the KSE would immediately violate the maximum principle. In contrast, the model introduced in this present paper allows for such a  parameter $\epsilon>0$, which we call the ``calming parameter.''  In particular, by adjusting the nonlinear term in the \eqref{KSE}, we create a globally well-posed PDE that approximates solutions to the 2D KSE to arbitrary precision, at least on the time interval of existence and uniqueness of solutions to the KSE. Perhaps surprisingly, our construction does not require the use of a maximum principle, nor does it add artificial viscosity to the system.

The $N$-dimensional Kuramoto-Sivashinsky equation (KSE) is given in scalar form by
\begin{align}
\label{KSE_scalar}
\dt{\phi}+\tfrac{1}{2}|\nabla\phi|^2 + \triangle\phi + \triangle^2\phi &= 0.
\end{align}
with periodic boundary conditions on a domain $[0,L]^N$.
By setting $\bu=\nabla\phi$ in \eqref{KSE_scalar}, one formally\footnote{\label{not_grad} We do not claim that $\grad\phi$ is a unique solution to \eqref{KSE} when $\phi$ is a solution to \eqref{KSE_scalar}. We only observe that one can formally obtain the set of equations \eqref{KSE} by taking the gradient of equation \eqref{KSE_scalar}.  In particular, it may be the case that there exist solutions to \eqref{KSE} that are not gradients of solutions to \eqref{KSE_scalar}, or of any other function.}
obtains the vector formulation of KSE:
\begin{align}
\label{KSE}
\dt{\bu}+(\bu\cdot\nabla)\bu + \triangle\bu + \triangle^2\bu &= \mathbf{0},
\end{align}


These equations were originally proposed in the 1970's by Kuramoto and Tsuzuki in the studies of crystal growth \cite{Kuramoto_Tsuzuki_1975,Kuramoto_Tsuzuki_1976} as well as by Sivashinsky in the study of flame-front instabilities \cite{Sivashinsky_1977} (see also \cite{Sivashinsky_1980_stoichiometry}). It has since found many other applications in the sciences, such as describing the flow of fluid down inclined planes  \cite{sivashinsky1980vertical}, and has shown to be a generic feature of many physical phenomena involving bifurcations \cite{misbah1994secondary}. 

The study of the $1$D equation since its conception has been fruitful, and the equation has been shown to be rich with interesting dynamics. It is globally well-posed  \cite{Nicolaenko_Scheurer_1984,Tadmor_1986}, solutions continue to exhibit chaotic dynamics at large times (see, e.g., \cite{Collet_Eckmann_Epstein_Stubbe_1993_Analyticity,Hyman_Nicolaenko_1986,Michelson_Sivashinsky_1977_numerical,Nicolaenko_Scheurer_Temam_1985,Papageorgiou_Smyrlis_1991_KSE_chaos}),
and a large body of work has been published on quantitative results pertaining to the global attractor (see, e.g., \cite{Collet_Eckmann_Epstein_Stubbe_1993_Attractor,Collet_Eckmann_Epstein_Stubbe_1993_Analyticity,Constantin_Foias_Nicolaenko_Temam_1989,Constantin_Foias_Nicolaenko_Temam_1989_IM_Book,Foias_Nicolaenko_Sell_Temam_1985,Foias_Sell_Temam_1985,Foias_Sell_Titi_1989,Goluskin_Fantuzzi_2019,Goodman_1994,Grujic_2000_KSE,Hyman_Nicolaenko_1986,Ilyashenko_1992,Nicolaenko_Scheurer_Temam_1986,Robinson_2001,Tadmor_1986,Temam_1997_IDDS,Kostianko_Titi_Zelik_2018}).


There are far fewer results on the KSE in the $2$D case. Global well-posedness for sufficiently small initial data was first shown in \cite{Sell_Taboada_1992} on a domain $[0, 2\pi] \times [0, 2\pi \epsilon]$ with $\epsilon>0$ sufficiently small. This result was improved upon in \cite{Molinet_2000} by showing global existence on a domain $[0,L_1] \times [0,L_2]$ with $L_2 \leq C L_1^q$ for some particular $q$. Later works continued to improve on the sharpness of this bound (see, e.g.,  \cite{Molinet_2000_2D_BS, Massatt_2022_DCDSB, Benachour_Kukavica_Rusin_Ziane_2014_JDDE_2DKSE,Kukavica_Massatt_2023_JDDE} and references therein).
Other works employ control of the domain size as a means to control the instability in Fourier modes. It was shown in \cite{Ambrose_Mazzucato_2018} that for small enough domains (on which no growing Fourier modes are present in the linear terms), global existence holds when the initial data is sufficiently small in a certain Wiener algebra. This result was then extended in \cite{Ambrose_Mazzucato_2021} to domains in which there is one linearly growing mode in each direction. Further studies have investigated modified equations \cite{Boling_Fengqiu_1993_JPDE,Feng_Mazzucato_2021,CotiZelati_Dolce_Feng_Mazzucato_2021,Larios_Yamazaki_2020_rKSE,Tomlin_Kalogirou_Papageorgiou_2018,Molinet_2000_2D_BS} or have looked at the equations with different boundary conditions \cite{Galaktionov_Mitidieri_Pokhozhaev_2008,Larios_Titi_2015_BC_Blowup,Pokhozhaev_2008}. For other results on the case $N>1$, see also \cite{Biswas_Swanson_2007_KSE_Rn, Cao_Titi_2006_KSE, Larios_Rahman_Yamazaki_2021_JNLS_KSE_PS}.

The intent of the present work is to propose a modification of the $2$D KSE in vector form which is globally well-posed for any size of domain or initial data. To do this, we make use of what we call an \textit{algebraic calming function} or simply a \textit{calming function}\footnote{Such a function is simply a bounded smooth truncation function, but we call it a ``calming'' function due to the way it is used in the nonlinearity to suppress the algebraic growth of the nonlinear term.  We do not call it a ``regularization,'' since we reserve this term for techniques which smooth the equations by modifying derivative operators.} which constrains the advective velocity of the solution. 

We propose the following modification of system \eqref{KSE}.
\begin{subequations}\label{cKSE}
\begin{align}
\label{cKSE_EQ}
\dt{\bu}+(\boldeta^\epsilon(\bu)\cdot\nabla)\bu + \triangle\bu + \triangle^2\bu &= \mathbf{0},
\\
\label{cKSE_IC}
\bu(\bx,0)&=\bu_0(\bx),
\end{align}
\end{subequations}
with L-periodic\footnote{Note: One could easily consider rectangular non-square periodic domains, say $\nR^2/((L_1\nZ)\times(L_2\nZ))$ as well with slight modification of the techniques we use here.  For the sake of keeping the discussion focused, we do not pursue such matters here.} boundary conditions on the $2$-dimensional periodic torus $\nT^2:=\nR^2/(L\nZ)^2 = [0,L]^2$ for some $L>0$.  
We call $\epsilon>0$ the \textit{calming parameter}, and $\boldeta^\epsilon$ the \emph{calming function}. We require that $\bee$ satisfies the conditions described in \ref{eta_def}.  
For the sake of concreteness, we consider several example choices for $\boldeta^\epsilon$; namely

\begin{minipage}{0.48\textwidth}
\begin{align} \label{eta_choices}
\boldeta^\epsilon(\bx) = 
\begin{cases}
  \boldeta^\epsilon_1(\bx) &:= \frac{\bx}{1+\epsilon|\bx|}\\
  \boldeta^\epsilon_2(\bx) &:= \frac{\bx}{1+\epsilon^2|\bx|^2}\\
  \boldeta^\epsilon_3(\bx) &:= \frac{1}{\epsilon}\arctan(\epsilon\bx)
\end{cases}
\end{align}
\end{minipage}
\begin{minipage}{0.48\textwidth}
\begin{tikzpicture}[baseline=(theAxis.west)]
\begin{axis}[
  name=theAxis,
  axis equal image,
  hide axis,
  ymin=-8,ymax=10,
  domain=-13:13,
  samples=50,
  no markers,
  cycle list name=exotic,
  width=1\textwidth,
  xticklabels ={},
  yticklabels ={},
]
  \draw[-, thick] (-20,0)--(20,0); 
  \draw[-, thick] (0,-10)--(0,10); 
\addplot[thick] {x} node[left,pos=0.83]{$\,\,y=x$};
\addplot {x/(1+0.2*abs(x))} node[right,pos=0.97]{$\eta_1$};
\addplot {x/(1+0.04*x*x)} node[right,pos=0.97]{$\eta_2$};
\addplot {0.017453292519943295*atan(0.2*x)/0.2} node[right,pos=0.97]{$\eta_3$};
\end{axis}
\end{tikzpicture}
\end{minipage}

\noindent
where the arctangent acts componentwise; that is, for a given vector $\bz=\binom{z_1}{z_2}$, we denote $\arctan(\bz)=\binom{\arctan(z_1)}{\arctan(z_2)}$. 


\begin{remark}
    For a given choice of $\bee_i$, $i=1,2,$ or $3$, $\epsilon$ can be thought of as a parameter which limits the velocity scale advecting the flow. 
    However, it is not immediately clear that velocity scales are limited in the same way between choices of $\bee_i$ for a fixed $\epsilon$.
    One could imagine trying to find a more meaningful comparison by rescaling the $\bee_i$ to have the same supremum. However, in doing so, $\bee_i(\bx)$ is no longer a good pointwise approximation for $\bx$. Therefore, it may not be meaningful to compare different convergence rates (or errors) between different types for fixed $\epsilon$. Yet, for the sake of convenience, we plot all error curves, etc., on the same plot. Moreover, the goal of the present work is not to compare different types of $\bee_i$ but rather to exhibit the robustness of this approach to different choices of calming function $\bee_i$.
    \end{remark}

\begin{remark}
We see no major difficulty in extending our work to the case of physical boundary conditions, i.e., $\bu\big|_{\partial\Omega}=\triangle\bu\big|_{\partial\Omega}=\mathbf{0}$.  However, for the sake of simplicity, we only consider periodic boundary conditions in the present work.
\end{remark}
    

Subsection \ref{subsec_Main_Results} lists our main definitions and results, and Section \ref{secPre} lists some preliminaries.  Section \ref{sec_cKSE_well_posedness} contains a proof of global well-posedness, which is mostly standard Galerkin methods, but with some subtle differences due to the non-polynomial form of the nonlinearity.  Section \ref{sec_smooth_ini_data} contains a proof of higher-order (but not arbitrary order) regularity of solutions.  Section \ref{Convergence} contains a proof of convergence of solutions of the calmed equation to solutions to the original KSE as the calming parameter $\epsilon\rightarrow 0$.  The proof here is not so straight-forward due to issues with commutator terms involving the calming function. As we will see, these issues are circumvented by taking advantage of structural properties of the calming function, and then using a boot-strapping argument in time. In addition, our techniques yield an explicit convergence rate.  In Section \ref{sec_cKSE_scal} we extend our ideas to the scalar form of the KSE.  In particular, we consider the system 
\eqref{KSE_scalar},
\begin{align}
\label{cKSE_scalar}
\dt{\phi}+\tfrac{1}{2}(\boldeta^\epsilon(\nabla\phi)\cdot\nabla)\phi + \triangle\phi + \triangle^2\phi &= 0,
\\
\phi(\bx,0)&=\phi_0(\bx).
\end{align}
Section \ref{sec_Computational_Results} exhibits results from simulations and provides computational evidence that the convergence rates we obtained in Section \ref{Convergence} are sharp (at least, in terms of convergence order).  Concluding remarks are in Section \ref{sec_conclusions}.

\subsection{Main results}\label{subsec_Main_Results}

\begin{definition}\label{eta_def}
    We say $\bee: \R^2 \to \R^2$ is a \emph{calming function} if the following conditions hold:
\begin{enumerate}[label=(\roman*)]
        \item \label{eta_cond_Lip} $\bee$ is Lipschitz continuous with Lipschitz constant $1$,
        \item \label{eta_cond_bdd} For $\epsilon > 0$ fixed, $\bee$ is bounded.
\end{enumerate}
These two conditions are sufficient to show that \eqref{cKSE} is globally well-posed. In Section \ref{Convergence}, we impose this third condition to obtain convergence:
\begin{enumerate}[resume,label=(\roman*)]
        \item \label{eta_cond_conv} There exists $C > 0$, $\alpha >0$ and $\beta > 0$ such that for any $\bx \in \R^2$,
        \begin{align}\label{pwise_conv}
        \abs{\bee(\bx) - \bx} \leq C \epsilon^\alpha \abs{\bx}^\beta.
        \end{align}
\end{enumerate}
\end{definition}

\begin{proposition}\label{eta_cond_prop}
    Consider $\bee_i$ as described in \eqref{eta_choices}. 
    Then $\bee_i$ satisfies Conditions \ref{eta_cond_Lip}, \ref{eta_cond_bdd}, and \ref{eta_cond_conv} of Definition \ref{eta_def} for each $i =1,2,3$. In particular, the following explicit bounds hold for $\epsilon>0$.
    \begin{enumerate}[label=(\roman*)]
        \item For $\bee_1$, 
        \[
        \normLp{\infty}{ \bee_1} = \frac{1}{\epsilon}
        \text{ and } \abs{\bee_1(\bx) - \bx} \leq \epsilon \abs{\bx}^2.
        \]
        \item For $\bee_2$, 
        \[
        \normLp{\infty}{ \bee_2} = \frac{1}{2\epsilon}
        \text{ and } \abs{\bee_2(\bx) - \bx} \leq \epsilon^2 \abs{\bx}^3.
        \]
        \item For $\bee_3$, 
        \[
        \normLp{\infty}{ \bee_3} = \frac{\pi}{2\epsilon}
        \text{ and } \abs{\bee_3(\bx) - \bx} \leq \epsilon^2 \abs{\bx}^3.
        \]
    \end{enumerate}
\end{proposition}

 \begin{proof}
     Straightforward computations (using a Taylor series expansions in the case $i=3$) yield the result.
 \end{proof}

\begin{definition}\label{cKSE_vec_weakDef}
Let $\bu_0\in L^2(\nT^2)$ and let $T>0$.  We say that $\bu$ is a \textit{weak solution} to calmed KSE \eqref{cKSE} on the interval $[0,T]$ if $\bu\in L^2([0,T]; H^2(\nT^2))\cap C([0,T]; L^2(\nT^2))$, $\partial_t\bu\in L^2(0,T; H^{-2}(\nT^2))$, and $\bu$ satisfies \eqref{cKSE_EQ} in the sense of $L^2(0,T; H^{-2}(\nT^2))$ and satisfies \eqref{cKSE_IC} in the sense of $C([0,T]; L^2(\nT^2))$.
\end{definition}

\begin{theorem} [Global Well-Posedness]
 Let $\bu_0\in L^2(\nT^2)$, let $T>0$ and fix $\epsilon>0$. Suppose $\bee$ is a calming function which satisfies Conditions \ref{eta_cond_Lip} and \ref{eta_cond_bdd} of Definition \ref{eta_def}. Then weak solutions to \eqref{cKSE} on $[0,T]$ exist, are unique, and depend continuously on the initial data in $L^\infty(0,T; L^2(\nT^2)) \cap L^2(0,T; H^2(\nT^2))$.
\end{theorem}

\begin{theorem} [Regularity]
Suppose that $\bee$ is is calming function which satisfies Conditions \ref{eta_cond_Lip}, and \ref{eta_cond_bdd} of \ref{eta_def}.
   Let $m\in\{1,2\}$, and suppose that $\bu$ is a weak solution to \eqref{cKSE} on $[0,T]$ for some $T>0$. If $\bu_0 \in H^m(\T)$, then $\bu \in L^\infty(0,T; H^m(\nT^2))\cap L^2(0,T;H^{m+2}(\T))$.
\end{theorem}

\begin{theorem} [Convergence]
Given $\bu_0 \in L^2(\T)$, let		
    \begin{align}
			\bu \in C([0, T]; L^2(\T)) \cap L^2(0, T; H^2(\T)).
            \label{rr}
		\end{align}
    be the corresponding weak solution of \eqref{KSE} with maximal time of existence and uniqueness $T^*>0$ and with $T\in(0,T^*)$.
    Suppose $\bee$ satisfies Conditions \ref{eta_cond_Lip} and \ref{eta_cond_bdd} of Definition \ref{eta_def}. Furthermore, suppose $\bee$ satisfies Condition \ref{eta_cond_conv}, so that \eqref{pwise_conv} holds for some fixed $C, \alpha > 0$ and any $\beta \in [0,3]$.
    Let $\bu^\epsilon$ be the corresponding weak solution of \eqref{cKSE} with calming function $\bee$ and initial data $\bu_0$. 
    Then for any $\epsilon>0$, it holds that 
    \begin{align*}
        \| \bu^\epsilon - \bu\|_{L^\infty(0,T;L^2)} & \leq 
        K \epsilon^\alpha, \\
        \| \bu^\epsilon - \bu\|_{L^2(0,T;H^2)} & \leq 
        K' \epsilon^\alpha, 
    \end{align*}
    where $K, K' >0$ depend on $T$, $\beta$, and various norms of $\bu$, but not on $\epsilon$ or $\alpha$.
\end{theorem}

\begin{definition}\label{cKSE_scal_weakDef1}
Let $\phi_0\in L^2(\nT^2)$ and let $T>0$.  We say that $\phi$ is a \textit{weak solution} to \eqref{cKSE_scal} on the interval $[0,T]$ if $\phi\in L^2([0,T]; H^2(\nT^2))\cap C([0,T]; L^2(\nT^2))$, $\partial_t\phi\in L^2(0,T; H^{-2}(\nT^2))$, and $\phi$ satisfies \eqref{cKSE_scal_EQ} in the sense of $L^2(0,T; H^{-2}(\nT^2))$ and satisfies \eqref{cKSE_scal_IC} in the sense of $C([0,T]; L^2(\nT^2))$.
\end{definition}

\begin{theorem} [Global Well-posedness in scalar form]
    \label{cKSE_scal_Ex_and_Uni1} 
    Let initial data $\phi_0 \in L^2(\mathbb{T}^2)$ be given, and let $T>0$,  $\epsilon > 0$ be fixed. Suppose $\bee$ is a calming function which satisfies Conditions \ref{eta_cond_Lip} and \ref{eta_cond_bdd} of Definition \ref{eta_def}.
    Then weak solutions to \eqref{cKSE_scal} on $[0,T]$ exist, are unique, and depend continuously on the initial data in $L^\infty(0,T; L^2(\nT^2)) \cap L^2(0,T; H^2(\nT^2))$.
\end{theorem} 

\begin{theorem} [Convergence in scalar form]
    \label{thm_cKSE_scal_conv1}
    Choose $\phi_0 \in L^2(\T)$ and let $\phi$ be the corresponding weak solution of the scalar KSE \eqref{KSE_scalar} with maximal time of existence $T^*$. 
    We assume that $\phi$ is in the natural energy space: for $T < T^*$,
		\begin{align}
			\phi \in C([0, T]; L^2(\T)) \cap L^2(0, T; H^2(\T)).
            \label{rr_scal}
		\end{align}
    Suppose $\bee$ satisfies \ref{eta_cond_Lip}, \ref{eta_cond_bdd}, and \ref{eta_cond_conv} of Definition \ref{eta_def}, so that there exists $C$, $\alpha >0$, and $\beta \in [0, 3 ]$ for which \eqref{pwise_conv} holds.
    and let $\phi^\epsilon$ be the corresponding weak solution of the scalar calmed KSE \eqref{cKSE_scal} with calming function $\bee$ and with initial data $\phi_0$. 
    Consider the convergence of $\phi^\epsilon$ to $\phi$ on the interval $[0,T]$.
    The difference $\phi^\epsilon - \phi$ satisfies
    \begin{align*}
        \| \phi^\epsilon - \phi\|_{L^\infty(0,T;L^2)} & \leq 
        K \epsilon^\alpha, \\
        \| \phi^\epsilon - \phi\|_{L^2(0,T;H^2)} & \leq 
        K' \epsilon^\alpha, 
    \end{align*}
    where $K, K' >0$ depend on $T$, $\beta$, and various norms of $\phi$, but not on $\epsilon$ or $\alpha$.
\end{theorem}

\section{Preliminaries}\label{secPre}
\noindent
In this section, we lay out some notation and preliminary notions that will be used later in the text. We denote the $2\pi$-periodic box $\nT^2:=\nR^2/(2\pi\nZ)^2=[0,2\pi)^2$.  
We denote the set of real vector-valued $L^2$ functions on $\nT^2$ by
\begin{align*}
L^2(\nT^2):=\set{
    \bu \bigg|\bu(\bx)=\sum_{\bk\in\nZ^2} \hat{\bu}_\bk  e^{i\bk\cdot\bx},  \,
    \cnj{\hat{\bu}_\bk} = \hat{\bu}_{-\bk}, \text{ and }\sum_{\bk\in\nZ^2}|\hat{\bu}_\bk|^2<\infty
}
\end{align*}  
(with the usual convention of equivalence up to sets of measure zero).
We also denote the (real) $L^2$ inner-product and $H^s$ Sobolev norm, $s\in \R$,  by
\[
 (\bu,\bv) := \sum_{i=1}^2\int_{\nT^2} u_i(\bx) v_i(\bx)\,d\bx,
\qquad
 \|\bu\|_{H^s} := \pnt{\sum_{\bk\in\nZ^2} \lp 1+ |\bk|\rp ^{2s}|\hat{\bu}_\bk|^{2}}^{1/2},
\]
and the corresponding space $H^s(\nT^2) = \set{\bu\in L^2(\nT^2)\big| \quad \|\bu\|_{H^s}<\infty}$.
The  space $L^2(\T)$ has an orthogonal basis of eigenfunctions of the operator Laplacian operator $-\lap$ given by 
\begin{align*}
     \left\{ \lp e^{i\bk\cdot\bx}, 0 \rp, \lp 0, e^{i\bk\cdot\bx} \rp  : \bk\in\nZ^2 \right\},
 \end{align*}
 with corresponding eigenvalues $\left\{ \abs{\bk}^2 : \bk\in\nZ^2 \right\}$.

For any $m\in\nN$, we denote by $P_m: L^2(\nT^2) \to L^2(\nT^2)$ the projection onto finitely many eigenfunctions of the operator $-\lap$, 
\[
 P_m \bu = \sum_{\substack{\bk\in\nZ^2\\|\bk|\leq m}}\hat{\bu}_\bk e^{i\bk\cdot\bx}.
\]
Denote $Q_m:=I-P_m$.  We recall the following projection estimates for any $\bu\in H^s(\T)$, $s>0$,
\begin{align}
    \normLp{2}{(-\lap)^s P_m \bu} &\leq m^s \normLp{2}{P_m\bu}  
    \label{Pm_Est} \\ 
    \|Q_m\bu\|_{L^2} &\leq
    \frac{1}{m^{s}}\|\bu\|_{H^s}.
    \label{QNproj}
\end{align}



We also recall Agmon's inequality on $\nT^2$, for $s_1< 1 <s_2$, 
\begin{align}\label{agmon}
\|\bu\|_{L^\infty}\leq C\|\bu\|_{H^{s_1}}^{\theta}\| \bu \|_{H^{s_2}}^{1-\theta},
\end{align}
where $\theta s_1 + (1-\theta)s_2 = 1$, and Ladyzhenskaya's inequality,
\begin{align}\label{ladyz}
\normLp{4}{\bu} \leq C \normLp{2}{\bu}^{1/2} \normLp{2}{\grad\bu}^{1/2}.
\end{align}

Using integration by parts, the Cauchy-Schwarz, and Young's inequalities, we obtain, for any $\delta>0$, the estimate
\begin{align}
    \label{grad_interp}
    \normLp{2}{\grad\bu}^2 \leq 
    \frac{1}{2\delta}\normLp{2}{\bu}^2 + 
    \frac{\delta}{2}\normLp{2}{\lap\bu}^2
\end{align}
which will be used repeatedly throughout the paper.
We also denote by $C$ a positive constant which may change from line to line.

 \section{Global Well-Posedness for Calmed KSE}\label{sec_cKSE_well_posedness}

\begin{lemma}\label{eta_lemma_Lip}
Suppose that $\bee$ satisfies Conditions \ref{eta_cond_Lip} and \ref{eta_cond_bdd} of \ref{eta_def}.  Then the following statements hold.
\begin{enumerate}[label=(\roman*)]
    \item Given $1 \leq p \leq \infty$, if $\bu \in L^p(\nT^2)$ then $\bee(\bu)\in L^p(\nT^2)$ and $\bee$ is Lipschitz in $ L^p(\nT^2)$ with Lipschitz constant $1$.
    \item Fix $\bu, \bw\in L^2(0,T; L^2(\nT^2))$ and $T>0$, let \newline $I_{\bu,\bw}: L^2(0,T; H^1(\nT^2)) \to \nR$ 
    be the map 
    \begin{align}\label{LinFunc}
        I_{\bu,\bw}(\bphi) = \int_0^T \lp \lp \bee(\bu)\cdot \grad \rp \bphi, \bw \rp dt .
    \end{align}
    Then $I_{\bu,\bw}$ is continuous.
\end{enumerate}
\end{lemma}

\begin{proof}
    \textit{(i)}. The result follows immediately from the definition of the $L^p$ norm and from Condition \ref{eta_cond_Lip} of Definition \ref{eta_def}.

   \textit{(ii)}. Let $\eta_j^\epsilon(\bu)$ denote the $j$-th component of $\bee(\bu)$. \\ For $\bphi \in L^2(0,T; H^1(\nT^2))$, we estimate
    \begin{align*}
        \abs{ I_{\bu,\bw}(\bphi) } &\leq 
        \sum_{j=1}^2 \int_0^T  \abs{ \lp \eta_j^\epsilon(\bu)\partial_j \bphi, \bw \rp }dt \\ &\leq 
        \sum_{j=1}^2 \int_0^T   \normLp{\infty}{\eta_j^\epsilon(\bu)}\normLp{2}{\partial_j \bphi} \normLp{2}{\bw}  dt \\ &\leq 
        \normLp{\infty}{\bee} \int_0^T \normHs{1}{\bphi} \normLp{2}{\bw}  dt \\ &\leq 
        \normLp{\infty}{\bee} \norm{\bw}_{L^2(0,T;L^2)} \norm{\bphi}_{L^2(0,T; H^1)}
    \end{align*}
    by the Cauchy-Schwarz inequality. 
\end{proof}


Using the projection operator $P_m$, define the finite-dimensional space $H_m := P_m(L^2(\nT^2))$.
Consider the following initial value problem obtained via Galerkin approximation: 
Given $\bu_0 \in L^2(\nT^2)$, find $\bu\in H_m$ which satisfies 
\begin{subequations}\label{cKSE_Galerkin}
\begin{align}
\dt{\bu} + P_m\lp(\boldeta^\epsilon(\bu)\cdot\nabla)\bu\rp + \triangle\bu + \triangle^2\bu &= \mathbf{0},
\\
\bu(\bx,0)&=P_m\bu_0(\bx).
\end{align}
\end{subequations}



\begin{lemma}\label{LocalLipschitz}
    If $\bee$ satisfies \ref{eta_cond_Lip} of Definition \ref{eta_def}, then the map $F: H_m \to H_m$ defined by
    \[ F(\bu)= -P_m\lp(\boldeta^\epsilon(\bu)\cdot\nabla)\bu\rp - \triangle\bu - \triangle^2\bu\]
    is locally Lipschitz on $H_m$. As a consequence, solutions to \eqref{cKSE_Galerkin} exist and are unique in $C^1([0,T],H_m)$ for some $T >0$.  
\end{lemma}

\begin{proof}
    Fix $\bu \in H_m$ and let $\bv \in H_m$ be arbitrary. Rewrite the difference $F(\bu) - F(\bv)$ as 
    \begin{align}
        F(\bu) - F(\bv) &= 
        - \lap\lp \bu - \bv \rp 
        - \lap^2\lp \bu - \bv \rp 
        - P_m \lp \lp \lp \bee(\bu) - \bee(\bv) \rp \cdot \grad \rp \bu \rp \notag \\ &\quad  
        - P_m \lp \lp \bee(\bv)\cdot \grad \rp \lp \bu - \bv \rp \rp.
        \notag
    \end{align}
    From Condition \ref{eta_cond_Lip} of Definition \ref{eta_def}, Estimate \eqref{Pm_Est}, and Agmon's inequality, it follows that
    \begin{align*}
    &\quad
        \normLp{2}{F(\bu) - F(\bv)} 
        \\&\leq 
        \normLp{2}{\lap\lp \bu - \bv \rp } +
        \normLp{2}{\lap^2\lp \bu - \bv \rp} \\ &\quad + 
        \normLp{2}{\lp \lp \bee(\bu) - \bee(\bv) \rp \cdot \grad \rp \bu} + 
        \normLp{2}{\lp \bee(\bv)\cdot \grad \rp \lp \bu - \bv \rp } \\ &\leq 
        \lp m+m^2 \rp \normLp{2}{\bu - \bv} +
        \normLp{\infty}{\grad\bu }\normLp{2}{\bee(\bu) - \bee(\bv)} \\ &\quad + 
        \normLp{\infty}{\bee(\bv)} \normLp{2}{\grad \lp \bu - \bv  \rp}  \\ &\leq 
        \lp m+m^2 \rp \normLp{2}{\bu - \bv} +
        \normHs{3}{\bu }\normLp{2}{\bu - \bv} + 
        m^\frac{1}{2}\normLp{\infty}{\bee}\normLp{2}{\bu - \bv}.
    \end{align*}
    Since $\bu$ is a finite linear combination of eigenfunctions of $-\lap$, $\normHs{3}{\bu} < \infty$. Thus $F$ is locally Lipschitz at $\bu \in H_m$.  Existence and uniqueness of solutions to \eqref{cKSE_Galerkin}  in $C^1([0,T],H_m)$ now follows as a consequence of the Picard-Lindel\"of Theorem
\end{proof}


\begin{remark}
    Due to the presence of the calming function $\bee$, the Galerkin system here is not necessarily quadratic such as in the case of the 2D Navier-Stokes equations or the 2D Kuramoto-Sivashinsky equations. Thus we verify the well-posedness of the system.
\end{remark}

\begin{theorem} \label{Ex_and_Uni}
 Let $\bu_0\in L^2(\nT^2)$, let $T>0$ and fix $\epsilon>0$. Suppose $\bee$ is a calming function which satisfies Conditions \ref{eta_cond_Lip} and \ref{eta_cond_bdd} of Definition \ref{eta_def}. Then weak solutions to \eqref{cKSE} on $[0,T]$ exist, are unique, and depend continuously on the initial data in $L^\infty(0,T; L^2(\nT^2)) \cap L^2(0,T; H^2(\nT^2))$.
\end{theorem}

\begin{proof} 

First, we show that a solution exists using Galerkin approximation.
Given $\bu_0 \in L^2(\nT^2)$, suppose $\bu^m \in C([0,T_m]; H_m)$
is a solution to \eqref{cKSE_Galerkin} on the interval $[0,T_m]$ for some $T_m>0$ with initial data $\bu_0^m = P_m\bu_0$. We take the inner product of \eqref{cKSE_Galerkin} with $\bu^m$ to obtain
\begin{align*}
    \frac{1}{2}\frac{d}{dt}\normLp{2}{\bu^m}^2 +
    \normLp{2}{\lap\bu^m}^2 &= -
    \lp \lap \bu^m , \bu^m \rp - 
    \lp (\bee\lp\bu^m\rp\cdot\nabla)\bu^m , \bu^m \rp
\end{align*}
We estimate the first term by $-\lp \lap \bu^m , \bu^m \rp \leq \frac{1}{4} \normLp{2}{\lap\bu^m}^2 +   \normLp{2}{\bu^m}^2$. For the nonlinear term, we estimate
\begin{align*}
    |\lp \bee\lp\bu^m\rp\cdot\nabla)\bu^m , \bu^m \rp|
    &\leq 
    \normLp{\infty}{\bee(\bu^m)}
    \normLp{2}{\nabla \bu^m}
    \normLp{2}{\bu^m}
    \\ &\leq 
    \normLp{\infty}{\bee}
    \normLp{2}{\bu^m}^\frac{1}{2}
    \normLp{2}{\lap\bu^m}^\frac{1}{2}
    \normLp{2}{\bu^m}
    \\ & \leq 
    \frac{3}{4}
    \normLp{\infty}{\bee}^\frac{4}{3}
    \normLp{2}{\bu^m}^2 +
    \frac{1}{4}
    \normLp{2}{\lap\bu^m}^2
\end{align*}
Combining the above estimates and denoting $K_{\epsilon} := \frac{3}{2} \normLp{\infty}{\bee}^\frac{4}{3} + 2$, we obtain
\begin{align}\label{gal_est1}
    \frac{d}{dt}\normLp{2}{\bu^m}^2 +
    \normLp{2}{\lap\bu^m}^2 
    &\leq
    K_{\epsilon}
    \normLp{2}{\bu^m}^2.
\end{align}
After dropping the second term of \eqref{gal_est1}, Gr\"onwall's inequality yields for all $t \in [0, T_m]$,
\begin{align}
    \normLp{2}{\bu^m(t)}^2 \leq 
    e^{K_{\epsilon} t}
    \normLp{2}{\bu^m(0)}^2 \leq 
    e^{K_{\epsilon} T_m}
    \normLp{2}{\bu_0}^2.
\end{align}


Since $\bu^m \in C([0,T_m], \nT^2)$, via a bootstrapping argument, 
it holds that for any $T>0$ and any $t\in[0,T]$, 
\begin{align}\label{gal_Linf_L2}
    \normLp{2}{\bu^m(t)}^2 \leq 
    e^{K_{\epsilon} t}
    \normLp{2}{\bu_0}^2
    \leq 
    e^{K_{\epsilon} T}
    \normLp{2}{\bu_0}^2.
\end{align}
Next, we integrate \eqref{gal_est1} on $[0,T]$ and apply Estimate \eqref{gal_Linf_L2}:
\begin{align}
    &\quad 
    \normLp{2}{\bu^m(T)}^2 + 
    \frac{1}{2}\int_0^T \normLp{2}{\lap \bu^m(s)}^2 ds
    \\ &\leq \notag
    \int_0^T K_{\epsilon} \normLp{2}{\bu^m(s)}^2ds + 
    \normLp{2}{\bu^m(0)}^2
    \\ &\leq \notag
    \int_0^T K_{\epsilon} e^{K_{\epsilon} s}
    \normLp{2}{\bu_0}^2 ds + 
    \normLp{2}{\bu_0}^2
    \\ &= \notag
     e^{K_{\epsilon} T}
    \normLp{2}{\bu_0}^2.
\end{align}
%
Hence, for all $T>0$,
\begin{align}\label{LiL2_L2H2_Galerkin_bound}
    \set{\bu^m}_{m=1}^\infty\text{ is bounded in }L^\infty([0,T];L^2)\cap
    L^2([0,T];H^2).
\end{align}
To bound the time derivative, we estimate
\begin{align*}
    \normHs{-2}{\dt \bu^m} 
    &\leq 
    \normHs{-2}{\lap^2\bu^m} 
    +
    \normHs{-2}{\lap\bu^m}
    + 
    \sup_{\substack{ \phi\in H^2 \\ \normHs{2}{\phi} =1 }}{\abs{\left\langle  P_m\lp(\boldeta^\epsilon(\bu^m)\cdot\nabla)\bu^m \rp, \phi \right\rangle}}
    \\ &\leq 
    C_1\normHs{2}{\bu^m}
    +
    C_2\normLp{2}{\bu^m}
    +
    \sup_{\substack{ \phi\in H^2 \\ \normHs{2}{\phi} =1 }}{\abs{\left\langle  \boldeta^\epsilon(\bu^m)\cdot\nabla)\bu^m , P_m\lp \phi \rp \right\rangle}}
    \\ &\leq 
    C\normHs{2}{\bu^m}
    +
    C\normLp{2}{\bu^m}
    +
    \sup_{\substack{ \phi\in H^2 \\ \normHs{2}{\phi} =1 }}
    { 
    \normLp{\infty}{\boldeta^\epsilon(\bu^m)}
    \normHs{1}{\bu^m}\normLp{2}{\phi}
    }
    \\ &\leq
    C\normHs{2}{\bu^m}
    +
    C\normLp{2}{\bu^m}
    +
    \normLp{\infty}{\bee}
    \normHs{1}{\bu^m}.
\end{align*}
Hence, $\set{\partial_t\bu^m}_{m=1}^\infty$ is bounded in $L^2(0,T; H^{-2}(\nT^2))$. 
By the Banach-Alaoglu Theorem, there exists 
$\bu \in L^2(0,T; H^2(\mathbb{T}^2)) \cap L^\infty(0,T; L^2(\mathbb{T}^2))$ 
and a subsequence (which we will still label as $\bu^m$) such that
\begin{align}
    \label{BA_sseq_0}         
    \bu^m &\rightharpoonup^* \bu \text{ weak-* in } L^\infty(0,T; L^2(\mathbb{T}^2)), \\
    \label{BA_sseq_1}         
    \bu^m &\rightharpoonup \bu \text{ weakly in } L^2(0,T; H^2(\mathbb{T}^2)), \\
    \label{BA_dt_sseq}      
    \dt{}\bu^m &\rightharpoonup \dt{}\bu \text{ weakly in } L^2(0,T; H^{-2}(\nT^2)).
\end{align}
Moreover, by the Aubin-Lions Lemma we may pass to another subsequence, relabeled to be $\bu^m$, such that 
\begin{align}
    \label{AL_sseq} \bu^m &\to \bu \text{ strongly in } C(0,T; L^2(\mathbb{T}^2)).
\end{align}
Now we can show that $\bu$ is a weak solution to \eqref{cKSE}. Given $\bw\in L^2(0,T;H^2(\mathbb{T}^2))$, we compute
\begin{align*}
    &\quad\lp\left\langle \dt{} \bu , \bw \right\rangle +
    \lp (\boldeta^\epsilon(\bu)\cdot\nabla)\bu, \bw \rp +
    \lp \triangle\bu, \bw \rp +
    \lp \triangle \bu, \lap \bw \rp\rp \\ &\qquad -
    \lp\left\langle \dt{} \bu^m , \bw \right\rangle +
    \lp P_m\lp(\boldeta^\epsilon(\bu^m)\cdot\nabla)\bu^m\rp, \bw \rp +
    \lp \triangle\bu^m, \bw \rp +
    \lp \triangle \bu^m, \lap \bw \rp \rp 
    \\ &=
    \left\langle \dt{} \lp \bu -\bu^m \rp , \bw \right\rangle +
    \lp \triangle\lp\bu - \bu^m\rp, \bw \rp +
    \lp \triangle \lp\bu - \bu^m\rp, \lap \bw \rp \\ &\qquad +
    \lp (\boldeta^\epsilon(\bu)\cdot\nabla)\bu, \bw \rp -
    \lp P_m\lp(\boldeta^\epsilon(\bu^m)\cdot\nabla)\bu^m\rp, \bw \rp
    \\ &=
    \left\langle \dt{} \lp \bu -\bu^m \rp , \bw \right\rangle +
    \lp \triangle\lp\bu - \bu^m\rp, \bw \rp +
    \lp \triangle \lp\bu - \bu^m\rp, \lap \bw \rp \\ &\qquad +
    \lp (\boldeta^\epsilon(\bu)\cdot\nabla)\bu, \bw \rp -
    \lp (\boldeta^\epsilon(\bu^m)\cdot\nabla)\bu^m, \bw \rp +
    \lp Q_m\lp(\boldeta^\epsilon(\bu^m)\cdot\nabla)\bu^m\rp, \bw \rp
    \\ &=
    \left\langle \dt{} \lp \bu -\bu^m \rp , \bw \right\rangle +
    \lp \triangle\lp\bu - \bu^m\rp, \bw \rp +
    \lp \triangle \lp\bu - \bu^m\rp, \lap \bw \rp \\ &\qquad +
    \lp (\lp\boldeta^\epsilon(\bu) - \boldeta^\epsilon(\bu^m) \rp\cdot\nabla)\bu^m, \bw \rp +
    \lp (\boldeta^\epsilon(\bu)\cdot\nabla)\lp\bu - \bu^m \rp, \bw \rp \\ &\qquad +
    \lp \lp(\boldeta^\epsilon(\bu^m)\cdot\nabla)\bu^m\rp, Q_m\bw \rp.
    \\ &:=
    \sum_{k=1}^6 I_k.
\end{align*}
Integrate $\sum_{k=1}^6 I_k$ in time for $t\in [0, T]$. We observe that $I_1, I_2,$ and $I_3$ all vanish as $m\to \infty$ by \eqref{BA_sseq_0}, \eqref{BA_sseq_1}, and \eqref{BA_dt_sseq}. 
Using Condition \ref{eta_cond_Lip} of Definition \ref{eta_def}, Agmon's inequality, Ladyzhenskaya's inequality, and H\"older's inequality, we obtain
\begin{align} \label{GC_I4}
    \int_0^T I_4 dt  &\leq 
    \int_0^T
    \normLp{2}{\bee(\bu) - 
    \bee(\bu^m)}
    \normLp{2}{\grad \bu^m}
    \normLp{\infty}{\bw}
    dt
     \\ &\leq 
    C \int_0^T 
    \normLp{2}{\bu - \bu^m}
    \normLp{2}{\bu^m}^\frac{1}{2}
    \normLp{2}{\lap \bu^m}^\frac{1}{2}
    \normLp{2}{\bw}^\frac{1}{2}
    \normHs{2}{\bw}^\frac{1}{2} dt \notag \\ &\leq 
    C \| \bu - \bu^m \|_{L^\infty(0,T;L^2)}^\frac{1}{2}
    \| \bu^m \|_{L^\infty(0,T;L^2)}^\frac{1}{2} 
    \notag \\ &\quad \times
    \int_0^T 
    \normLp{2}{\bu - \bu^m}^\frac{1}{2}
    \normLp{2}{\lap \bu^m}^\frac{1}{2}
    \normHs{2}{\bw} dt 
    \notag \\ &\leq 
    C \| \bu - \bu^m \|_{L^\infty(0,T;L^2)}^\frac{1}{2}
    \| \bu^m \|_{L^\infty(0,T;L^2)}^\frac{1}{2} 
    \notag \\ \notag&\quad \times 
    \| \lap \bu^m \|_{L^2(0,T;L^2)}^\frac{1}{2} 
    \| \bw\|_{L^2(0,T;H^2)}
    \| \bu - \bu^m \|_{L^2(0,T;L^2)}^\frac{1}{2}, 
\end{align}
which is bounded due to \eqref{gal_Linf_L2}, \eqref{BA_sseq_1}, and \eqref{AL_sseq}. \\
For $I_5$, 
\begin{align} \label{GC_I5}
    \int_0^T I_5 dt = I_{\bu,\bw}(\bu - \bu^m)    
\end{align}
for $I_{\bu, \bw}$ as defined in \eqref{LinFunc}, which convergences due to Lemma \ref{eta_lemma_Lip}. Finally, using H\"older's inequality, Condition \ref{eta_cond_bdd} of Definition \ref{eta_def}, and \eqref{QNproj},
\begin{align} \label{GC_I6}
    \int_0^T I_6 dt
    &\leq
    \int_0^T \normLp{\infty}{\bee(\bu^m)} \normLp{2}{\grad \bu^m} \normLp{2}{Q_m\bw} dt 
    \\ & \leq
    \normLp{\infty}{\bee}
    \lp \int_0^T \normLp{2}{\grad \bu^m}^2 dt \rp^{\frac{1}{2}}
    \lp \int_0^T \normLp{2}{Q_m\bw}^2 dt \rp^{\frac{1}{2}} 
    \notag \\ & \leq
    \normLp{\infty}{\bee} \lp \int_0^T \normLp{2}{\grad \bu^m}^2 dt \rp^{\frac{1}{2}}
    \lp \int_0^T \frac{1}{m^4} \normHs{2}{\bw}^2 dt \rp^{\frac{1}{2}}
    \notag \\ & \leq
    \normLp{\infty}{\bee} \| \bu^m \|_{L^2(0,T;H^2)} 
    \|{\bw}\|_{L^2(0,T;H^2)}  \frac{1}{m^2}, 
    \notag
\end{align}
which converges to zero by \eqref{LiL2_L2H2_Galerkin_bound}. \\
Invoking \eqref{BA_sseq_1}, \eqref{BA_dt_sseq}, \eqref{AL_sseq}, \eqref{GC_I4}, \eqref{GC_I5}, and \eqref{GC_I6}, 
\begin{align*}
 \lim_{m\to\infty} \int_0^T\lp \sum_{k=1}^6 I_k\rp dt = 0.
\end{align*}
Therefore solutions to the ODE system \eqref{cKSE_Galerkin} converge to a solution of the PDE system \eqref{cKSE}.
Thus $\bu$ is indeed a solution to $\eqref{cKSE}$. 

Now we show that the solution $\bu$ satisfies $\bu(0) = \bu_0$ in the sense of $C([0,T],L^2)$.
Applying Lemma 1.1 from Chapter 3 of \cite[p. 250]{Temam_2001_Th_Num}, for all $\bv\in H^2(\nT^2)$, it follows that
   \begin{align}\label{dist_eq}
       \la \dt\bu , \bv \ra = 
       \frac{d}{dt}\lp \bu , \bv \rp = -
       \lp \lap\bu , \bv \rp -
       \lp \lap\bu, \lap \bv \rp -
       \lp \bee(\bu)\cdot\grad\bu, \bv \rp
   \end{align}
in the scalar distribution sense on $[0,T]$. Now, suppose that $\psi\in C^1([0,T])$ and satisfies $\psi(0) = 1$,  $\psi(T) = 0$. We then integrate \eqref{dist_eq} in time with $\psi$ and apply integration by parts to obtain
   \begin{align} \label{ini_data_check0}
    \int_0^T \lp \bu , \bv \rp \psi'(t) dt = &-
    \int_0^T \lp \lap\bu , \bv \rp \psi(t) dt -
    \int_0^T \lp \lap\bu, \lap \bv \rp \psi(t) dt \\ &-
    \int_0^T \lp \bee(\bu)\cdot\grad\bu, \bv \rp \psi(t)  dt + \lp \bu(0), \bv\rp. \notag
\end{align}
On the other hand, if we take the inner product of \eqref{cKSE_Galerkin} with $\bv$ then integrate in time with $\psi$ we obtain
\begin{align*}
    \int_0^T \lp \bu^m , \bv \rp \psi'(t) dt = &-
    \int_0^T \lp \lap\bu^m , \bv \rp \psi(t) dt -
    \int_0^T \lp \lap\bu^m, \lap \bv \rp \psi(t) dt 
    \\ & 
    -\int_0^T \lp P_m \lp \bee(\bu^m)\cdot\grad\bu^m \rp, \bv \rp \psi(t)  dt + \lp \bu^m_0, \bv\rp.
\end{align*}
Passing to the limit as $m\rightarrow\infty$ then yields
\begin{align} \label{ini_data_check1}
    \int_0^T \lp \bu , \bv \rp \psi'(t) dt = &-
    \int_0^T \lp \lap\bu , \bv \rp \psi(t) dt -
    \int_0^T \lp \lap\bu, \lap \bv \rp \psi(t) dt \\ &-
    \int_0^T \lp \bee(\bu)\cdot\grad\bu, \bv \rp \psi(t)  dt + \lp \bu_0, \bv\rp. \notag
\end{align}
By then comparing \eqref{ini_data_check0} and \eqref{ini_data_check1}, we obtain $\lp \bu(0) - \bu_0, \bv \rp = 0$ for all $\bv\in H^2(\nT^2)$. Since $H^2(\nT^2)$ is dense in $L^2(\nT^2)$, it follows that $\lp \bu(0) - \bu_0, \bv \rp = 0$ for all $\bv\in L^2(\nT^2)$. Thus $\bu$ satisfies $\bu(0) = \bu_0$.
Next, we show that weak solutions are unique. Set $\bw = \bu - \bv$, where $\bu$ and $\bv$ are both weak solutions of calmed KSE $\eqref{cKSE}$ on the interval $[0,T]$ with $\bu_0 = \bv_0$. 
After taking the difference of the two equations, we then take the action of the difference equation with $w$, which yields
\begin{align} \label{diff_est1}
    &\quad
    \frac{1}{2}\frac{d}{dt}\normLp{2}{\bw}^2 +
    \normLp{2}{\lap\bw}^2 
    \\&= - \notag
    \lp \lap \bw, \bw\rp -
    \lp \lp\bee\lp\bu\rp\cdot\grad \rp \bu , \bw \rp +
    \lp \lp\bee\lp\bv\rp\cdot\grad \rp \bv , \bw \rp 
    \\ &= \notag
    -\lp \lap \bw, \bw\rp +
    \lp \lp \lp \bee\lp\bv\rp - \bee\lp\bu\rp \rp \cdot\grad \rp \bu , \bw \rp -
    \lp \lp\bee\lp\bv\rp\cdot\grad \rp \bw , \bw \rp 
    \\ &= J_1 + J_2 + J_3, \notag
\end{align}
where we have used the Lions-Magene Lemma to write 
$\left\langle \dt \bw, \bw \right\rangle = \frac{1}{2}\frac{d}{dt}\normLp{2}{\bw}^2$. Then,
\begin{align*}
    J_1 \leq 
    \frac{1}{2}\normLp{2}{\bw}^2 +
    \frac{1}{2}\normLp{2}{\lap \bw}^2.
\end{align*}
Also,
\begin{align*}
    J_2 &:=   \lp \lp \lp \bee\lp\bv\rp - \bee\lp\bu\rp \rp \cdot\grad \rp \bu , \bw \rp \\ &\leq 
    \normLp{4}{\bee(\bv) - \bee(\bu)}
    \normLp{2}{\grad \bu}
    \normLp{4}{\bw} \\ &\leq 
    C \normLp{2}{\grad \bu} \normLp{4}{\bw}^2 \\ &\leq 
    C \normLp{2}{\grad \bu} \normLp{2}{\bw}\normLp{2}{\grad\bw} \\ &\leq 
    C \normLp{2}{\grad \bu} \normLp{2}{\bw}^\frac{3}{2}\normLp{2}{\lap\bw}^\frac{1}{2} \\ &\leq
    C\normLp{2}{\grad \bu}^\frac{4}{3}\normLp{2}{\bw}^2 + \frac{1}{4}\normLp{2}{\lap\bw}^2
\end{align*}
by Agmon's inequality, Ladyzhenskaya's inequality, and Lemma \ref{eta_lemma_Lip}. Finally, 
\begin{align*}
    J_3 &:= -
    \lp \lp\bee\lp\bv\rp\cdot\grad \rp \bw , \bw \rp 
    \\ &\leq 
    \normLp{\infty}{\bee(\bv)}
    \normLp{2}{\grad \bw} 
    \normLp{2}{\bw}
    \\ &\leq 
    \lp
    \normLp{\infty}{\bee}
    \normLp{2}{\bw}^\frac{3}{2}
    \rp \lp
    \normLp{2}{\lap \bw}^\frac{1}{2}
    \rp
    \\ &\leq 
    \frac{3}{4}
    \normLp{\infty}{\bee}^\frac{4}{3}
    \normLp{2}{\bw}^2 + 
    \frac{1}{4}
    \normLp{2}{\lap \bw}^2
\end{align*}
using Young's inequality and \eqref{grad_interp}.
From the above estimates, we obtain 
\begin{align*}
    \frac{d}{dt}\normLp{2}{\bw(t)}^2 \leq 
    \lp
    1 + C \normLp{2}{\grad\bu}^\frac{4}{3} + \frac{3}{2}\normLp{\infty}{\bee} ^\frac{4}{3}
    \rp \normLp{2}{\bw(t)}^2.
\end{align*}
Writing $K_1(t) =1 + C \normLp{2}{\grad\bu}^\frac{4}{3} + \frac{3}{2}\normLp{\infty}{\bee} ^\frac{4}{3}$, we observe that $K(t)$ is integrable on $[0,T]$. Thus we conclude, recalling that $\bw = \bu-\bv$,
\begin{align}\label{Gronwall_uniqueness1}
    \normLp{2}{\bu(t)-\bv(t)}^2 \leq 
    \normLp{2}{\bu_0-\bv_0}^2 \exp{\lp\int_0^T K_1(t) dt\rp} .
\end{align}
Therefore solutions to \eqref{cKSE} are unique. 
If we now integrate \eqref{diff_est1} on the interval $[0,T]$ and apply estimate \eqref{Gronwall_uniqueness1}, we obtain
\begin{align}\label{Gronwall_uniqueness2}
    \int_0^T \normLp{2}{\lap \bu(t)- \lap \bv(t)}^2  dt \leq 
    K_2\normLp{2}{\bu_0-\bv_0}^2
\end{align}
For some constant $K_2$ depending on $T$, $\normLp{2}{\grad\bu}$, and $\normLp{\infty}{\bee}$. From estimates \eqref{Gronwall_uniqueness1} and \eqref{Gronwall_uniqueness2} we conclude that solutions depend continuously on the initial data in $L^\infty(0,T; L^2(\nT^2)) \cap L^2(0,T; H^2(\nT^2))$.
\end{proof}


 \section{Higher-Order Regularity of Solutions}\label{sec_smooth_ini_data}
\noindent
 In this section, we only work formally, but the results can be made rigorous by using, e.g., the Galerkin method.
We will show that the regularity of a weak solution $\bu$ to \ref{cKSE} is dependent on the regularity of the calming function $\bee$ and the initial data $\bu_0$.




\begin{theorem}
   Suppose that $\bee$ is is calming function which satisfies Conditions \ref{eta_cond_Lip}, and \ref{eta_cond_bdd} of Definition \ref{eta_def}.
   Let $m\in\{1,2\}$, and suppose that $\bu$ is a weak solution to \eqref{cKSE} on $[0,T]$ for some $T>0$. If $\bu_0 \in H^m(\T)$, then $\bu \in L^\infty(0,T; H^m(\nT^2))\cap L^2(0,T;H^{m+2}(\T))$.
\end{theorem}

\begin{remark}
    It seems likely that higher-order regularity ($m>2$) also holds, but we do not pursue such matters here.
\end{remark}

\begin{proof}
    We first show the case $m=1$.
    We take the (formal) inner product of \eqref{cKSE} with $-\lap\bu$ and integrate by parts to obtain
    \begin{align*}
    \lp \dt{\bu},-\lap\bu \rp - 
    \lp (\boldeta^\epsilon(\bu)\cdot\nabla)\bu, \lap\bu\rp -
    \lp \lap\bu, \lap\bu \rp -
    \lp \lap^2\bu, \lap\bu \rp = 0
    \end{align*}
    which we will rewrite as 
    \begin{align*}
        \frac{1}{2}\frac{d}{dt}\normLp{2}{\grad\bu}^2 + 
        \normLp{2}{\grad\lap\bu}^2 &= 
        \lp \lp(\boldeta^\epsilon(\bu)\cdot\nabla)\bu \rp, \lap\bu\rp - 
        \lp \grad\bu, \grad\lap\bu \rp \\ &=
        \lp(\boldeta^\epsilon(\bu)\cdot\nabla)\bu, \lap\bu\rp - 
        \lp \grad\bu, \grad\lap\bu \rp.
    \end{align*}
    Thus, we obtain 
    \begin{align*}
    &\quad
        \frac{1}{2}\frac{d}{dt}\normLp{2}{\grad\bu}^2 + 
        \normLp{2}{\grad\lap\bu}^2 
        \\ &\leq \notag
        \normLp{\infty}{\boldeta^\epsilon(\bu)}
        \normLp{2}{\grad\bu}\normLp{2}{\lap\bu} +
        \normLp{2}{\grad\bu}\normLp{2}{\grad\lap\bu} 
        \\ &\leq \notag
        \normLp{\infty}{\bee} 
        \normLp{2}{\grad\bu}^\frac{3}{2}
        \normLp{2}{\grad\lap\bu}^\frac{1}{2} 
        +
        \normLp{2}{\grad\bu}\normLp{2}{\grad\lap\bu} 
        \\ &\leq \notag
        \lp \frac{3}{4} \normLp{\infty}{\bee}^{\frac{4}{3}} + \frac{1}{2} \rp \normLp{2}{\grad\bu}^2 + 
        \frac{3}{4}\normLp{2}{\grad\lap\bu}^2.
    \end{align*}
    This estimate can then be rewritten as 
    \begin{align}\label{H1_bd1}
        \frac{d}{dt}\normLp{2}{\grad\bu}^2 + 
        \frac{1}{2}\normLp{2}{\grad\lap\bu}^2 \leq 
        \lp
        \frac{3}{2}\normLp{\infty}{\bee}^\frac{4}{3} + 1
        \rp \normLp{2}{\grad\bu}^2.
    \end{align}
    Then, by Gr\"onwall's inequality, 
    \begin{align}\label{H1_bd2}
        \normLp{2}{\grad\bu(t)}^2 
       &\leq
        \normLp{2}{\grad\bu_0}^2
        \exp{\lp \frac{3}{2} t \normLp{\infty}{\bee}^{\frac{4}{3}} + t\rp}  \\&\leq\notag
        \normLp{2}{\grad\bu_0}^2
        \exp{\lp \frac{3}{2} T \normLp{\infty}{\bee}^{\frac{4}{3}} + T\rp}
    \end{align}
    Now, after integrating \eqref{H1_bd1} on the interval $[0,T]$ and applying Estimate \eqref{H1_bd2}, it follows that
    \begin{align}
        \int_0^T\normLp{2}{\grad\lap\bu(\tau)}^2  d\tau \leq
        2 \normLp{2}{\grad\bu_0}^2
        \exp{\lp \frac{3}{2} T \normLp{\infty}{\bee}^{\frac{4}{3}} + T\rp}.
    \end{align}
    Thus, $\bu \in L^2(0,T;H^3(\nT^2))\cap L^\infty(0,T;H^1(\T))$ whenever $\bu_0 \in H^1(\nT^2)$. \\

    The case $m = 2$ proceeds in a similar way. We take the inner product with $\lap^2 \bu$ to obtain 
    \begin{align}
    &\quad
        \frac{1}{2}\frac{d}{dt}\normLp{2}{\lap \bu}^2 + \normLp{2}{\lap^2\bu}^2 
        \\& \leq 
        \abs{\lp \lap\bu, \lap^2\bu \rp } +
        \abs{\lp \lp\bee(\bu)\cdot \grad\rp\bu , \lap^2 \bu \rp} \notag \\ &\leq 
        \frac{1}{2}\normLp{2}{\lap \bu}^2 +
        \frac{1}{2}\normLp{2}{\lap^2 \bu}^2 + 
        \normLp{4}{\bee(\bu)}\normLp{4}{\grad\bu}\normLp{2}{\lap^2 \bu }  \notag \\ &\leq 
        \frac{1}{2}\normLp{2}{\lap \bu}^2 +
        \frac{1}{2}\normLp{2}{\lap^2 \bu}^2 +
        C\normLp{4}{\bu}^2\normLp{4}{\grad\bu}^2 + 
        \frac{1}{4}\normLp{2}{\lap^2 \bu}^2  \notag \\ &\leq 
        \frac{1}{2}\normLp{2}{\lap \bu}^2 +
        \frac{3}{4}\normLp{2}{\lap^2 \bu}^2 +
        C\normLp{2}{\bu}\normLp{2}{\grad\bu}^2\normLp{2}{\lap\bu} \notag 
        \\ &\leq \notag 
        \lp \frac{1}{2} + C\normLp{2}{\bu}^2 \rp
        \normLp{2}{\lap\bu}^2 +
        \frac{3}{4}\normLp{2}{\lap^2 \bu}^2.
    \end{align}
    Similar to the case $m=1$, this estimate reveals that 
    $\bu \in L^\infty(0,T; H^2(\nT^2))\cap L^2(0,T; H^4(\nT^2))$ whenever $\bu_0 \in H^2(\nT^2)$.
\end{proof}
 \section{Convergence to Kuramoto-Sivashinsky Solutions}\label{Convergence}
 
 
 \noindent
It is known that, for any initial data $\bu_0 \in L^2(\nT^2)$,  solutions to 2D KSE exist and are unique in $C([0, T]; L^2(\T)) \cap L^2(0, T; H^2(\T))$ for some (possibly only small) $T>0$ (see, e.g.,  \cite{Biswas_Swanson_2007_KSE_Rn,Feng_Mazzucato_2021}).
In this section we show that as $\epsilon\rightarrow 0$, solutions $\bu^\epsilon$ of the calmed KSE \eqref{cKSE} converge to solutions $\bu$ of KSE \eqref{KSE} prior to its potential blowup time. For this result, it seems necessary that our calming function $\bee$ satisfies Condition \ref{eta_cond_conv} of Definition \ref{eta_def}. Indeed, if one wants to show that $\lp\bee(\bu^\epsilon) \cdot \grad \rp \bu^\epsilon \to 
\lp \bu \cdot \grad \rp \bu$ in some sense as $\epsilon \to 0$, then one expects that at least $\bee(\bx) \to \bx$ as $\epsilon \to 0$. We do not find this imposition to be restrictive, as our example choices for $\bee$ satisfy this condition, as seen in Proposition \ref{eta_cond_prop}.

\begin{theorem}\label{conv}
    Given $\bu_0 \in L^2(\T)$, let		
    \begin{align}
			\bu \in C([0, T]; L^2(\T)) \cap L^2(0, T; H^2(\T)).
		\end{align}
    be the corresponding weak solution of \eqref{KSE} with maximal time of existence and uniqueness $T^*>0$ and with $T\in(0,T^*)$.
    Suppose $\bee$ satisfies \ref{eta_cond_Lip} and \ref{eta_cond_bdd} of Definition \ref{eta_def}. Furthermore, suppose $\bee$ satisfies Condition \ref{eta_cond_conv}, so that \eqref{pwise_conv} holds for some fixed $C, \alpha > 0$ and any $\beta \in [0,3]$.
    Let $\bu^\epsilon$ be the corresponding weak solution of \eqref{cKSE} with calming function $\bee$ and initial data $\bu_0$. 
    Then for any $\epsilon>0$, it holds that 
    \begin{align}
        \| \bu^\epsilon - \bu\|_{L^\infty(0,T;L^2)} & \leq 
        K \epsilon^\alpha, \\
        \| \bu^\epsilon - \bu\|_{L^2(0,T;H^2)} & \leq 
        K' \epsilon^\alpha, 
    \end{align}
    where $K, K' >0$ depend on $T$, $\beta$, and various norms of $\bu$, but not on $\epsilon$ or $\alpha$.
\end{theorem}

\vskip .1in 

In order to prove Theorem \ref{conv}, we need the following 
abstract bootstrapping/continuity argument (see, e.g., 
\cite[p. 20]{Tao_2006}).

\begin{lemma} \label{boot}
	Let $T>0$. Assume that two statements $C(t)$ and $H(t)$ with $t\in [0,T]$ 
	satisfy the
	following conditions:
	\begin{enumerate}
		\item[(a)] If $H(t)$ holds for some $t\in [0,T]$, then $C(t)$ holds for 
		the same $t$;
		\item[(b)] If $C(t)$ holds for some $t_0\in [0,T]$, then $H(t)$ holds 
		for $t$ in a neighborhood of $t_0$;
		\item[(c)] If $C(t)$ holds for $t_m\in [0,T]$ and $t_m \to t$, then 
		$C(t)$ holds;
		\item[(d)] $H(t)$ holds for at least one $t_1 \in [0,T]$.
	\end{enumerate}
	Then $C(t)$ holds for all $t\in[0,T]$.
\end{lemma}


With this lemma, we are now ready to prove our convergence result.

\begin{proof}[Proof of Theorem \ref{conv}]
Set 
\[
\bw^\epsilon :=\bu^\epsilon - \bu
\]
and take the difference between \eqref{cKSE} 
and \eqref{KSE} 
to obtain
\[
\p_t \bw^\epsilon + \Delta \bw^\epsilon+ \Delta^2 \bw^\epsilon =-  
(\bee (\bu^\epsilon) \cdot \na)\bu^\epsilon + (\bu\cdot\na) \bu.
\]
Testing each side by $\bw^\epsilon$ we obtain, after 
integration by parts,
\begin{align}
\frac12\frac{d}{dt} \|\bw^\epsilon\|_{L^2}^2 
+ \|\Delta \bw^\epsilon\|_{L^2}^2 = \|\na \bw^\epsilon\|_{L^2}^2 +  N, \label{root}
\end{align}
where $N$ is given by 
\[
N := - \int_{\T} ((\bee (\bu^\epsilon) \cdot 
\na)\bu^\epsilon - (\bu\cdot\na) \bu) 
\cdot \bw^\epsilon \,d\bx.
\]

By inequality \eqref{grad_interp}, 
\begin{align}
\|\na \bw^\epsilon\|_{L^2}^2 \le \|\bw^\epsilon\|_{L^2}\, \|\Delta 
\bw^\epsilon\|_{L^2} \le \frac12 \|\Delta \bw^\epsilon\|_{L^2}^2 + \frac12 
\|\bw^\epsilon\|_{L^2}^2.\label{int}
\end{align}
Inserting \ref{int} in \ref{root} yields
\begin{align}	
\frac{d}{dt} \|\bw^\epsilon\|_{L^2}^2 
+ \|\Delta \bw^\epsilon\|_{L^2}^2 \leq \|\bw^\epsilon\|_{L^2}^2 +  2N.
\label{root0}
\end{align}

$N$ can be written as 
\begin{align*}
    N = -
        &\int_{\T} ((\bee (\bu^\epsilon) -\bee (\bu)) \cdot \na) \bw^\epsilon \cdot \bw^\epsilon \,d\bx -
        \int_{\T}  (\bee (\bu)\cdot\na)\bw^\epsilon \cdot \bw^\epsilon \,d\bx 
        \\ -
        &\int_{\T}  ((\bee (\bu^\epsilon) - \bee (\bu ))\cdot \na)\bu \cdot \bw^\epsilon \,d\bx -
        \int_{\T}  ((\bee (\bu )-\bu)\cdot \na)\bu \cdot \bw^\epsilon \,d\bx.
\end{align*}
Using the Lipschitz property of $\bee$ and \eqref{pwise_conv}, we see that $N$ is bounded by 
\begin{align*}
    \abs{N} \leq \quad 
        &\int_{\T} \abs{\bee (\bu^\epsilon) -\bee (\bu) }\abs{\grad \bw^\epsilon} \abs{\bw^\epsilon} \,d\bx +
        \int_{\T}  \abs{\bee (\bu)} \abs{\grad \bw^\epsilon} \abs{\bw^\epsilon} \,d\bx \\ + 
        &\int_{\T} \abs{\bee (\bu^\epsilon) -\bee (\bu) } \abs{\grad \bu} \abs{\bw^\epsilon} \,d\bx +
        \int_{\T}  \abs{\bee (\bu )-\bu} \abs{\grad \bu} \abs{\bw^\epsilon} \,d\bx. \\ \leq \quad
        &\int_{\T} \abs{\bw^\epsilon }^2\abs{\grad \bw^\epsilon} \,d\bx +
        \int_{\T}  \abs{\bu} \abs{\grad \bw^\epsilon} \abs{\bw^\epsilon} \,d\bx \\ + 
        &\int_{\T} \abs{\bw^\epsilon }^2 \abs{\grad \bu} \,d\bx +
        C \epsilon^\alpha \int_{\T}  \abs{\bu}^\beta \abs{\grad \bu} \abs{\bw^\epsilon} \,d\bx. \\ = \quad 
        & N_1 + N_2 + N_3 + N_4.
\end{align*}

These terms can be bounded as follows. By H\"older's inequality, 
Ladyzhenskaya's inequality, \eqref{grad_interp}, and Young's inequality,
\begin{align}
\label{n1b}
	N_1 &\le \|\bw^\epsilon\|_{L^4}^2 \|\na \bw^\epsilon\|_{L^2} \le C\, 
	\|\bw^\epsilon\|_{L^2} \,\|\na \bw^\epsilon\|_{L^2}^2\\
	& \le C\, \|\bw^\epsilon\|_{L^2}^2\,\|\Delta \bw^\epsilon\|_{L^2} \le 
	\frac1{16} \|\Delta \bw^\epsilon\|_{L^2}^2 +C\,\|\bw^\epsilon\|_{L^2}^4. 
 \notag
\end{align}
\begin{align}
\label{n2b}
	N_2 &\le \|\bu\|_{L^2} \, \|\bw^\epsilon\|_{L^4} \|\na 
	\bw^\epsilon\|_{L^4}\\ 
	&\le C\, \|\bu\|_{L^2} \,
	\|\bw^\epsilon\|_{L^2}^\frac12 \,\|\na \bw^\epsilon\|_{L^2}\,\|\Delta 
	\bw^\epsilon\|_{L^2}^\frac12\notag\\
	& \le C\, \|\bu\|_{L^2} \,
	\|\bw^\epsilon\|_{L^2} \|\Delta 
	\bw^\epsilon\|_{L^2}\notag\\
	&\le 
	\frac1{16} \|\Delta \bw^\epsilon\|_{L^2}^2 +C\,\|\bu\|_{L^2}^2\, 
	\|\bw^\epsilon\|_{L^2}^2.\notag
\end{align}
\begin{align}
\label{n3b}
	N_3 &\le  \|\na \bu\|_{L^2} \,\|\bw^\epsilon\|_{L^4}^2 \le C\, 
	\|\bu\|_{L^2}^\frac12\,\|\Delta \bu\|_{L^2}^\frac12 
	\,\|\bw^\epsilon\|_{L^2}\,\|\na \bw^\epsilon\|_{L^2} \\
	&\le C\, 
	\|\bu\|_{L^2}^\frac12\,\|\Delta \bu\|_{L^2}^\frac12 
	\,\|\bw^\epsilon\|^\frac32 _{L^2}\,\|\Delta \bw^\epsilon\|_{L^2}^\frac12  
	\notag\\
	&\le\frac1{16} \|\Delta \bw^\epsilon\|_{L^2}^2 
	+C\,\|\bu\|_{L^2}^\frac23\,\|\Delta \bu\|_{L^2}^\frac23
	\,\|\bw^\epsilon\|^2_{L^2}.\notag
\end{align}

By H\"older's inequality, 
Ladyzhenskaya's inequality, and \eqref{grad_interp},

\begin{align}
\label{n4b}
    N_4 &\leq 
    C \epsilon^\alpha 
    \normLp{2}{\bw^\epsilon} 
    \normLp{\infty}{\bu}^\beta  
    \normLp{2}{\grad\bu} \\ &\leq 
    C \epsilon^\alpha 
    \normLp{2}{\bw^\epsilon} 
    \normLp{2}{\bu}^\frac{\beta}{2}  
    \normHs{2}{\bu}^\frac{\beta}{2} 
    \normLp{2}{\bu}^\frac{1}{2}
    \normLp{2}{\lap \bu}^\frac{1}{2} \notag \\ &\leq 
    C \epsilon^\alpha 
    \normLp{2}{\bw^\epsilon} 
    \normLp{2}{\bu}^\frac{\beta}{2} \lp
    \normLp{2}{\bu}^\frac{\beta}{2} +
    \normLp{2}{\lap \bu}^\frac{\beta}{2} \rp
    \normLp{2}{\bu}^\frac{1}{2}
    \normLp{2}{\lap \bu}^\frac{1}{2} \notag \\ &= 
    \lp
    C \epsilon^\alpha 
    \normLp{2}{\bu}^{\beta + \frac{1}{2}} 
    \normLp{2}{\lap \bu}^\frac{1}{2} +
    C \epsilon^\alpha 
    \normLp{2}{\bu}^{\frac{\beta}{2} + \frac{1}{2}} 
    \normLp{2}{\lap \bu}^{\frac{\beta}{2} + \frac{1}{2}}
    \rp 
    \normLp{2}{\bw^\epsilon}\notag
\end{align}

We now insert the bounds for $N$ in \eqref{root}:

\begin{align}
\label{en}
    & \frac{d}{dt} \normLp{2}{\bw^\epsilon}^2 + 
    \normLp{2}{\lap \bw^\epsilon}^2 \leq 
    \normLp{2}{\bw^\epsilon}^2 +  
    2N \\ \leq \quad & 
    \frac{6}{16}\normLp{2}{\lap \bw^\epsilon}^2 + 
    C\normLp{2}{\bw^\epsilon}^4 
    \notag \\ + &
    \lp
    1 + C \normLp{2}{\bu}^2 + 
    C\normLp{2}{\bu}^\frac{2}{3} \normLp{2}{\lap \bu}^\frac{2}{3}
    \rp
    \normLp{2}{\bw^\epsilon}^2 \notag \\  + & 
    C
    \lp
    \epsilon^\alpha 
    \normLp{2}{\bu}^{\beta + \frac{1}{2}} 
    \normLp{2}{\lap \bu}^\frac{1}{2} +
    \epsilon^\alpha 
    \normLp{2}{\bu}^{\frac{\beta}{2} + \frac{1}{2}} 
    \normLp{2}{\lap \bu}^{\frac{\beta}{2} + \frac{1}{2}}
    \rp 
    \normLp{2}{\bw^\epsilon} \notag
\end{align}

Due to the presence of the term $\|\bw^\epsilon\|_{L^2}^4$, we cannot apply Gr\"onwall's inequality directly.  However, since $\|\bw^\epsilon\|_{L^2}$ is supposed to be small,  this term is not a ``bad'' term and is even smaller than $\|\bw^\epsilon\|_{L^2}^2$. 
We just need to apply a bootstrapping 
argument, as stated in Lemma \ref{boot}. Denote by $H(t)$ with $t\in [0,T]$ the 
statement 
that 
\[
\|\bw^\epsilon(t)\|_{L^2} \leq 1
\]
and by $C(t)$ the statement that 
\[
\|\bw^\epsilon(t)\|_{L^2}  \leq e^{A(T)} B(T)\, \epsilon^\alpha \leq \tfrac12,
\]
where $A(t)$ and $B(t)$ are defined as in \eqref{BootIneq_A} and \eqref{BootIneq_B} below and $\epsilon$ 
is taken to be sufficiently small such that 
\[
e^{A(T)} B(T)\, \epsilon^\alpha \leq \tfrac12.
\]
Clearly, $C(t)$ is a stronger statement than $H(t)$, and thus (b) of Lemma 
\ref{boot} holds. When the solutions are regular enough, then 
$\|\bw^\epsilon(t)\|_{L^2}$ is continuous in time. Indeed, this regularity is given by Condition \ref{rr} and Definition \ref{cKSE_vec_weakDef} and thus (c) of Lemma 
\ref{boot} 
holds. For $t=0$, $\|\bw^\epsilon(t)\|_{L^2}$ is zero and thus (d) of  Lemma 
\ref{boot} holds.  In order to apply Lemma \ref{boot}, it remains to verify 
(a). 
That is, if $H(t)$ holds for some $t\in [0,T]$, namely
\[
\|\bw^\epsilon(t)\|_{L^2} \leq 1,
\]
then $C(t)$ holds at the same $t$, namely 
\[
\|\bw^\epsilon(t)\|_{L^2} \leq e^{A(T)} B(T)\, \epsilon^\alpha <\tfrac12.
\]
We assume that, for some $t\in [0,T]$, 
\begin{align}
\|\bw^\epsilon(t)\|_{L^2} \leq 1 \label{an}
\end{align}
and then show that \eqref{an} leads to a desired smaller bound at this 
same $t$.
Now we replace
$\|\bw^\epsilon\|_{L^2}^4$ by $\|\bw^\epsilon\|_{L^2}^2$ in \eqref{en} and eliminate $\|\bw^\epsilon\|_{L^2}$ from each term to obtain 

\begin{align}
    \frac{d}{dt} \normLp{2}{\bw^\epsilon} & \leq
    C
    \lp
    1 + \normLp{2}{\bu}^2 + 
    \normLp{2}{\bu}^\frac{2}{3} \normLp{2}{\lap \bu}^\frac{2}{3}
    \rp 
    \normLp{2}{\bw^\epsilon} \notag \\ \notag &\quad +
    C
    \epsilon^\alpha 
    \lp
    \normLp{2}{\bu}^{\beta + \frac{1}{2}} 
    \normLp{2}{\lap \bu}^\frac{1}{2} +
    \normLp{2}{\bu}^{\frac{\beta}{2} + \frac{1}{2}} 
    \normLp{2}{\lap \bu}^{\frac{\beta}{2} + \frac{1}{2}}
    \rp ,
\end{align}
in which we also use the fact that $\frac{d}{dt} \normLp{2}{\bw^\epsilon}^2 = 2 \normLp{2}{\bw^\epsilon} \frac{d}{dt} \normLp{2}{\bw^\epsilon}$. \vspace{2 mm}

Due to the regularity assumption on $\bu$ in \ref{rr}, the terms \\
$\lp 1 + \normLp{2}{\bu}^2 + \normLp{2}{\bu}^\frac{2}{3} \normLp{2}{\lap \bu}^\frac{2}{3} \rp$
and 
$\normLp{2}{\bu}^{\beta + \frac{1}{2}} \normLp{2}{\lap \bu}^\frac{1}{2}$
are integrable for $ \beta \geq 0$. Furthermore, for $\beta \leq 3$, 
$\frac{\beta}{2} + \frac{1}{2} \leq 2$ and thus
$\normLp{2}{\bu}^{\frac{\beta}{2} + \frac{1}{2}}\normLp{2}{\lap \bu}^{\frac{\beta}{2} + \frac{1}{2}}$
is integrable.
It then follows from Gr\"onwall's inequality that 
\begin{align}
\|\bw^\epsilon(t)\|_{L^2} \le e^{A(t)}\|\bw^\epsilon(0)\|_{L^2} +  e^{A(t)} B(t)\, \epsilon^\alpha \le e^{A(T)} B(T)\, \epsilon^\alpha, \label{bb}
\end{align}
where we have used the fact that the initial difference  $\bw^\epsilon(0)=0$ and have written 
\begin{align}\label{BootIneq_A}
A(t) &:=C\, \int_0^t 
\Big(
1 + \normLp{2}{\bu}^2 + 
\normLp{2}{\bu}^\frac{2}{3} \normLp{2}{\lap \bu}^\frac{2}{3}
\Big) ds,\\ \label{BootIneq_B}
B(t) &:= C\, \int_0^t
\Big(
\normLp{2}{\bu}^{\beta + \frac{1}{2}} 
\normLp{2}{\lap \bu}^\frac{1}{2} +
\normLp{2}{\bu}^{\frac{\beta}{2} + \frac{1}{2}} 
\normLp{2}{\lap \bu}^{\frac{\beta}{2} + \frac{1}{2}}
\Big)\,ds.
\end{align}

By taking $\epsilon$ sufficiently small, from \eqref{bb} we deduce that any \\ $t\in [0,T]$ which satisfies
\[
\|\bw^\epsilon(t)\|_{L^2} <1. 
\]
must also satisfy 
\[
\|\bw^\epsilon(t)\|_{L^2} \le  e^{A(T)} B(T)\, \epsilon^\alpha < \tfrac{1}{2}. 
\]
Thus the bootstrapping argument holds, and we conclude as claimed that for all $t \leq T$,
\begin{align} \label{conv_LinfL2_est}
\normLp{2}{\bw^\epsilon(t) } \leq K_1\epsilon^\alpha
\end{align}
where $K_1 = e^{A(T)} B(T)$ depends on $T$, $\bu$, and $\beta$. In particular, $\bu^\epsilon \to \bu$ in $L^\infty(0,T; L^2(\mathbb{T}^2))$ as $\epsilon\rightarrow0^+$. 
Next, integrate \eqref{en} for all $t\in[0,T]$ (again replacing $\normLp{2}{\bw^\epsilon}^4$ by $\normLp{2}{\bw^\epsilon}^2$) to obtain
\begin{align}
&\quad \notag
    \frac{10}{16}\int_0^T \normLp{2}{\lap \bw^\epsilon}^2 dt 
    \\&\leq \notag
    C \int_0^T
    \lp
    1 + \normLp{2}{\bu}^2 + 
    \normLp{2}{\bu}^\frac{2}{3} \normLp{2}{\lap \bu}^\frac{2}{3}
    \rp
    \normLp{2}{\bw^\epsilon}^2 \, dt  
    \\ & \quad + \notag
    C \epsilon^\alpha 
    \int_0^T \lp
    \normLp{2}{\bu}^{\beta + \frac{1}{2}} 
    \normLp{2}{\lap \bu}^\frac{1}{2} +
    \normLp{2}{\bu}^{\frac{\beta}{2} + \frac{1}{2}} 
    \normLp{2}{\lap \bu}^{\frac{\beta}{2} + \frac{1}{2}}
    \rp 
    \normLp{2}{\bw^\epsilon} \, dt
\end{align}
In which we are again using the fact that $\bw^\epsilon(0) = 0$.
Applying \eqref{BootIneq_A}, \eqref{BootIneq_B}, and \eqref{conv_LinfL2_est} then yields 
\begin{align}
    \int_0^T \normLp{2}{\lap \bw^\epsilon}^2 dt \leq 
    \frac{16}{10}A(T)K_1^2 \epsilon^{2\alpha} + 
    \frac{16}{10}B(T)K_1\epsilon^{2\alpha}.
\end{align}
For $K_2 = \frac{16}{10}A(T)K_1^2 + \frac{16}{10}B(T)K_1$ (again only depending on $T$, $\|\bu\|_{L^\infty(0,T;L^2)}$, $\|\bu\|_{L^2(0,T;H^2)}$, and $\beta$), we obtain 
\begin{align}
    \| \lap \bw^\epsilon \|_{L^2(0,T; L^2)} \leq K_2^\frac{1}{2} \epsilon^\alpha.
\end{align}
Using an interpolation inequality, we obtain
\begin{align}
\| \bw^\epsilon \|_{L^2(0,T; H^2)} &\leq
C\|\bw^\epsilon \|_{L^2(0,T; L^2)} +
C\|\lap \bw^\epsilon \|_{L^2(0,T; L^2)} \\ &\leq
CT^\frac{1}{2}\|\bw^\epsilon \|_{L^\infty(0,T; L^2)} +
C\|\lap \bw^\epsilon \|_{L^2(0,T; L^2)} \notag \\ &\leq
C\lp T^\frac{1}{2}K_1 + K_2^\frac{1}{2} \rp \epsilon^\alpha \notag
\end{align}
as claimed.  
In particular, $\bu^\epsilon \to \bu$ in $L^2(0,T; H^2(\nT^2))$ as $\epsilon\rightarrow0^+$.
\end{proof}

\begin{corollary}\label{conv_particular}
    Consider the calming functions $\bee$ as described in \eqref{eta_choices}. Let $\bu, \bu^\epsilon$ be as in the statement of Theorem \ref{conv} with the same initial data, where $\bu^\epsilon$ is determined by $\bee_i$, $i=1,2$, or $3$.  Then for $T<T^*$, there exists $K'_i>0$ independent of $\epsilon$ such that
    \begin{enumerate}
        \item for $\bee = \bee_1$,
        \begin{align}
        \label{conv_type1}
        \|\bu^\epsilon - \bu\|_{L^\infty(0,T;L^2)} \leq K'_1 \epsilon,
        \end{align}
        \item for $\bee = \bee_2,$
        \begin{align}
        \label{conv_type2}
        \|\bu^\epsilon - \bu\|_{L^\infty(0,T;L^2)} \leq K'_2 \epsilon^2,
        \end{align}
        \item for $\bee = \bee_3,$
        \begin{align}
        \label{conv_type3}
        \|\bu^\epsilon - \bu\|_{L^\infty(0,T;L^2)} \leq K'_3 \epsilon^2.
        \end{align}
    \end{enumerate}
\end{corollary}

\begin{proof}
    The proof follows immediately from Theorem \ref{conv} and Proposition \ref{eta_cond_prop}.
\end{proof}

 \section{The Scalar Form}\label{sec_cKSE_scal}
\noindent
Here we investigate the scalar formulation \eqref{cKSE_scalar}. The analysis is similar to the analysis of \eqref{cKSE}, so we only briefly present formal energy estimates. 

 We propose the following ``calmed'' version of the scalar KSE
\begin{subequations}\label{cKSE_scal}
\begin{align}
\label{cKSE_scal_EQ}
    \dt{\phi} + 
    \tfrac{1}{2}(\bee(\grad\phi)\cdot\grad)\phi + 
    \lap\phi + \lap^2\phi &= 0, \\
\label{cKSE_scal_IC}
    \phi(\bx,0)&=\phi_0(\bx).
\end{align}
\end{subequations}
Other formulations are of course possible.  For example, one could consider a nonlinearity of the form $\tfrac{1}{2}\bee(|\grad\phi|^2)$, or $\tfrac{1}{2}\bee(|\grad\phi|^{\delta})|\grad\phi|^{2-\delta}$ ($0<\delta<2)$, or $\frac{\frac12|\grad\phi|^2}{1+\epsilon^2|\phi|^2}$, or many other possibilities.  However, in the present work, we choose to focus only on the form in \eqref{cKSE_scal_EQ}, as the advective nature of the nonlinearity seems perhaps closest in spirit to the nature of the original equation.

\begin{definition}\label{cKSE_scal_weakDef}
Let $\phi_0\in L^2(\nT^2)$ and let $T>0$.  We say that $\phi$ is a \textit{weak solution} to \eqref{cKSE_scal} on the interval $[0,T]$ if $\phi\in L^2([0,T]; H^2(\nT^2))\cap C([0,T]; L^2(\nT^2))$, $\partial_t\phi\in L^2(0,T; H^{-2}(\nT^2))$, and $\phi$ satisfies \eqref{cKSE_scal_EQ} in the sense of $L^2(0,T; H^{-2}(\nT^2))$ and satisfies \eqref{cKSE_scal_IC} in the sense of $C([0,T]; L^2(\nT^2))$.
\end{definition}

\begin{theorem} 
    \label{cKSE_scal_Ex_and_Uni} 
    Let initial data $\phi_0 \in L^2(\mathbb{T}^2)$ be given, and let $T>0$,  $\epsilon > 0$ be fixed. Suppose $\bee$ is a calming function which satisfies Conditions \ref{eta_cond_Lip} and \ref{eta_cond_bdd} of Definition \ref{eta_def}.
    Then weak solutions to \eqref{cKSE_scal} on $[0,T]$ exist, are unique, and depend continuously on the initial data in $L^\infty(0,T; L^2(\nT^2)) \cap L^2(0,T; H^2(\nT^2))$.
\end{theorem} 
For the sake of brevity, we work formally rather than rigorously.  However, the proof below can be made rigorous, e.g., via the use of Galerkin methods as in the proof of Theorem \ref{Ex_and_Uni}.

\begin{proof}
    Take a (formal) inner product of \eqref{cKSE_scal} with $\phi$ and integrate by parts to obtain
    \begin{align}\label{scal_est1}
        \frac{1}{2}\frac{d}{dt}\normLp{2}{\phi}^2 +
        \normLp{2}{\lap\phi}^2 =  - \lp \lap \phi, \phi \rp - \lp \tfrac{1}{2}\bee\lp \grad \phi \rp \cdot \grad \phi , \phi \rp.
    \end{align}
    Using \eqref{grad_interp}  
    and 
   \begin{align}
   \notag
    \abs{\lp \tfrac{1}{2}\bee\lp \grad \phi \rp \cdot \grad \phi , \phi \rp} 
    &\leq 
    \frac{1}{2}\normLp{\infty}{\bee}\normLp{2}{\grad\phi}\normLp{2}{\phi} 
    \\&\notag
    \leq 
    C \normLp{\infty}{\bee}^\frac{4}{3}\normLp{2}{\phi}^2 +
    \frac{1}{4}\normLp{2}{\lap\phi}^2,
    \end{align}
    we obtain from \eqref{scal_est1} that 
    \begin{align}\label{scal_est2}
        \frac{d}{dt}\normLp{2}{\phi}^2 +
        \normLp{2}{\lap\phi}^2 \leq 
        \lp 2 + C \normLp{\infty}{\bee}^\frac{4}{3}\rp
        \normLp{2}{\phi}^2.
    \end{align}
    Hence from Gr\"onwall's inequality, dropping the second term in \eqref{scal_est2}, we obtain 
    \begin{align} \label{scal_LinfL2}
        \normLp{2}{\phi(t)}^2 \leq 
        e^{K_\epsilon T}\normLp{2}{\phi_0}^2,
    \end{align}
    where $K_\epsilon = \lp 2 + C \normLp{\infty}{\bee}^\frac{4}{3}\rp $. Hence $\phi \in L^\infty(0,T; L^2(\nT^2))$. Next, we integrate \eqref{scal_est2} in time on the interval $[0,T]$ and drop any unnecessary terms:
    \begin{align}
        \label{scal_L2H2}
        \int_0^T \frac{1}{2}\normLp{2}{\lap\phi}^2 &\leq 
        \int_0^T K_\epsilon
        \normLp{2}{\phi(t)}^2 dt + 
        \normLp{2}{\phi_0}^2  \\ &\leq 
        \int_0^T K_\epsilon
        e^{K_\epsilon T}\normLp{2}{\phi_0}^2 dt + 
        \normLp{2}{\phi_0}^2 \notag \\ &=  
        \lp K_\epsilon T
        e^{K_\epsilon T} + 1 \rp \normLp{2}{\phi_0}^2. \notag
    \end{align}
    Therefore $\phi \in L^\infty(0,T;L^2(\nT^2))\cap L^2(0,T;H^2(\nT^2))$. Now we obtain estimates on $\partial_t \phi$: For any $\psi\in L^2(0,T; H^2(\nT^2))$,
    \begin{align} \label{scal_dtEst}
        \abs{\left\langle\partial_t \phi, \psi \right\rangle} &= 
        \abs{\int_0^T \partial_t \phi  \psi dt}  \\ &=
        \abs{
        \int_0^T \lp \frac{1}{2}\bee\lp \grad \phi \rp \cdot \grad \phi \rp \psi dt +
        \int_0^T \lp \lap \phi \rp \psi dt +
        \int_0^T \lp \lap \phi \rp \lap \psi dt
        } \notag  \\ &\leq 
        \frac{1}{2}\int_0^T \abs{\bee\lp \grad \phi \rp}\abs{\grad \phi}\abs{ \psi} dt +
        \int_0^T \abs{\lap \phi }\abs{ \psi} dt +
        \int_0^T \abs{\lap \phi }\abs{\lap \psi} dt \notag \\ &\leq 
        \frac{1}{2}\normLp{\infty}{\bee }\|\grad \phi \|_{L^2(0,T;L^2)}\|\psi\|_{L^2(0,T;L^2)} 
        \notag \\ &\quad + 
        \|\lap \phi \|_{L^2(0,T;L^2)}\|\psi\|_{L^2(0,T;L^2)} + 
        \|\lap \phi \|_{L^2(0,T;L^2)}\|\lap \psi\|_{L^2(0,T;L^2)} \notag \\ &\leq 
        \lp
        \frac{1}{2}\normLp{\infty}{\bee }\|\phi \|_{L^2(0,T;H^2)} +
        2\|\phi \|_{L^2(0,T;H^2)} \rp \|\psi\|_{L^2(0,T;H^2)} . \notag
    \end{align}
    It follows from Estimate \eqref{scal_L2H2} that $\|\partial_t \phi \|_{L^2(0,T;H^{-2})} < \infty$, hence 
    $\partial_t \phi \in L^2(0,T;H^{-2}(\nT^2))$. From this we deduce that a solution $\phi$ to \eqref{cKSE_scalar} exists, with 
    \[
    \phi \in C(0,T; L^2(\nT^2)) \cap L^2(0,T; H^2(\nT^2)).
    \]
    Now, let $\phi$ and $\psi$ be two solutions to \eqref{cKSE_scalar} with $\phi(0) = \psi(0) = \phi_0$. Let $\delta = \phi - \psi$. Then $\delta$ satisfies the equation 
    \begin{align} \label{scal_diff}
        \partial_t \delta + \lap^2 \delta = - 
        \lap \delta +
        \bee(\grad\psi)\cdot \grad \psi -
        \bee(\grad\phi)\cdot \grad \phi
    \end{align}
    with $\delta(0) = 0$. We can then rewrite the nonlinear term as 
    \begin{align} \label{scal_nonlinearDiff}
        \bee(\grad\psi)\cdot \grad \psi -
        \bee(\grad\phi)\cdot \grad \phi = 
        \lp \bee(\grad\psi) - \bee(\grad\phi) \rp \cdot \grad \psi -
        \bee(\grad\phi) \cdot \grad \delta .
    \end{align}
    We now insert \eqref{scal_nonlinearDiff} into \eqref{scal_diff} and apply integration by parts to obtain
    \begin{align}\label{scal_en_est1}
    &\quad
        \frac{1}{2}\frac{d}{dt}\normLp{2}{\delta}^2 +
        \normLp{2}{\lap \delta}^2 
        \\&\leq \notag
        \abs{\lp \lap \delta, \delta \rp} +
        \abs{\lp \lp \bee(\grad\psi) - \bee(\grad\phi) \rp \cdot \grad \psi, \delta \rp} + 
        \abs{\lp \bee(\grad\phi) \cdot \grad \delta, \delta \rp}.
    \end{align}
    In the second term, we use Condition \ref{eta_cond_Lip} of \ref{eta_def}, H\"older's, Ladyzhenskaya's, and Young's inequality to obtain
    \begin{align}\label{scal_diff_est1}
        \abs{\lp \lp \bee(\grad\psi) - \bee(\grad\phi) \rp \cdot \grad \psi, \delta \rp} &\leq 
        \normLp{4}{\bee(\grad\psi) - \bee(\grad\phi)}\normLp{2}{\grad \psi} \normLp{4}{\delta} \\ &\leq 
        \normLp{2}{\grad \psi} \normLp{4}{\delta}\normLp{4}{\grad \delta} \notag \\ &\leq 
        C \normLp{2}{\grad \psi} \normLp{2}{\delta}^\frac{1}{2} 
        \normLp{2}{\grad \delta} \normLp{2}{\lap \delta}^\frac{1}{2} \notag \\ &\leq 
        C \normLp{2}{\grad \psi} \normLp{2}{\delta} \normLp{2}{\lap \delta} \notag \\ &\leq 
        C \normLp{2}{\grad \psi}^2 \normLp{2}{\delta}^2 + \frac{1}{4}\normLp{2}{\lap \delta}^2 \notag
    \end{align}
    In the third term, we apply Condition \ref{eta_cond_bdd} of \ref{eta_def}, use Young's inequality, and use interpolation inequalities to obtain 
    \begin{align}\label{scal_diff_est2}
        \abs{\lp \bee(\grad\phi) \cdot \grad \delta, \delta \rp} &\leq 
        \normLp{\infty}{\bee} \normLp{2}{\grad\delta}\normLp{2}{\delta}  \\ &\leq 
        \normLp{\infty}{\bee} \normLp{2}{\delta}^\frac{3}{2} \normLp{2}{\lap\delta}^\frac{1}{2} \notag \\ &\leq 
        C\normLp{\infty}{\bee}^\frac{4}{3} \normLp{2}{\delta}^2 + 
        \frac{1}{4}\normLp{2}{\lap\delta}^2. \notag
    \end{align}
    After inserting \eqref{scal_diff_est1} and \eqref{scal_diff_est2} into \eqref{scal_en_est1} and rearranging the terms, the inequality becomes
    \begin{align} \label{scal_en_est2}
        \frac{d}{dt}\normLp{2}{\delta}^2 +
        \normLp{2}{\lap \delta}^2 &\leq 
        C
        \lp
        1 + \normLp{2}{\grad \psi}^2 +
        \normLp{\infty}{\bee}^\frac{4}{3} 
        \rp \normLp{2}{\delta}^2.
    \end{align}
    Then applying Gr\"onwall's inequality, we obtain
    \begin{align}\label{scal_uni1}
        \normLp{2}{\phi(t) - \psi(t)}^2 \leq 
        e^{\widetilde{K}_1(T)} \normLp{2}{\phi_0 - \psi_0}^2,
    \end{align}
    where 
    $\widetilde{K}_1(T) = \int_0^T 1 + \normLp{2}{\grad \psi(t)}^2 + \normLp{\infty}{\bee}^\frac{4}{3} dt$. Since $\psi \in L^2(0,T; H^2(\nT^2))$, and $\bee$ is bounded, $\widetilde{K}_1(T) < \infty$.
    So $\phi(t) = \psi(t)$ for all $t\in [0,T]$, hence solutions to \eqref{cKSE_scalar} are unique.
    Now, we integrate \eqref{scal_en_est2} on the interval $[0,T]$ and apply \eqref{scal_uni1}, which yields 
    \begin{align}\label{scal_uni2}
        \int_0^T \normLp{2}{\lap \phi(t) - \lap \psi(t)}^2 dt \leq \widetilde{K}_2\normLp{2}{\phi_0 - \psi_0}^2
    \end{align}
    for some $\widetilde{K}_2$ which depends on $T$, $\normLp{2}{\grad \psi(t)}$, and $\normLp{\infty}{\bee}$. 
    From estimates \eqref{scal_uni1} and \eqref{scal_uni2} we conclude that solutions depend continuously on the initial data in $L^\infty(0,T; L^2(\nT^2)) \cap L^2(0,T; H^2(\nT^2))$.
\end{proof}

\begin{theorem}\label{cKSE_scal_conv}
    Given $\phi_0 \in L^2(\T)$, let $\phi$ be the corresponding weak solution of the scalar KSE \eqref{KSE_scalar} with maximal time of existence and uniqueness $T^*$. 
    We assume that $\phi$ is in the natural energy space: for $T \in (0,T^*)$,
		\begin{align}
			\phi \in C([0, T]; L^2(\T)) \cap L^2(0, T; H^2(\T)).
		\end{align}
    Suppose $\bee$ satisfies \ref{eta_cond_Lip}, \ref{eta_cond_bdd}, and \ref{eta_cond_conv} of Definition \ref{eta_def}, so that there exists $C$, $\alpha >0$, and $\beta \in [0, 3 ]$ for which \eqref{pwise_conv} holds.
   Let $\phi^\epsilon$ be the corresponding weak solution of the scalar calmed KSE \eqref{cKSE_scal} with calming function $\bee$ and with initial data $\phi_0$. 
    Consider the convergence of $\phi^\epsilon$ to $\phi$ on the interval $[0,T]$.
    The difference $\phi^\epsilon - \phi$ satisfies
    \begin{align*}
        \| \phi^\epsilon - \phi\|_{L^\infty(0,T;L^2)} & \leq 
        K \epsilon^\alpha, \\
        \| \phi^\epsilon - \phi\|_{L^2(0,T;H^2)} & \leq 
        K' \epsilon^\alpha, 
    \end{align*}
    where $K, K' >0$ depend on $T$, $\beta$, and various norms of $\phi$, but not on $\epsilon$ or $\alpha$.
\end{theorem}

\begin{proof}
    This proof has only minor variations from the proof of Theorem \ref{conv}. We set $\delta^\epsilon = \phi - \phi^\epsilon$ take the difference between \eqref{KSE_scalar} and \eqref{cKSE_scal_EQ}, and take the inner product with $\delta^\epsilon$. to obtain 
    \begin{align}
    \frac{d}{dt}\normLp{2}{\delta^\epsilon}^2 + \normLp{2}{\lap \delta^\epsilon}^2  \leq 
    \normLp{2}{\delta^\epsilon}^2 + N_1 + N_2 + N_3 + N_4,
    \label{cKSE_scal_conv1}
    \end{align}
    where 
    \begin{align*}
        N_1 &= 
        \abs{\lp  
            \lp \bee(\grad\phi^\epsilon) - \bee(\grad\phi) \rp 
            \cdot \grad \delta^\epsilon, \delta^\epsilon
        \rp} \leq 
        C \normLp{2}{\delta^\epsilon}^6 + \frac{3}{4}\normLp{2}{\lap \delta^\epsilon}^2, \\ 
        N_2 &= 
        \abs{\lp
            \lp \bee(\grad\phi^\epsilon) - \bee(\grad\phi) \rp 
            \cdot \grad \phi, \delta^\epsilon
        \rp} \leq 
        C\normLp{2}{\phi}^\frac{2}{5} \normLp{2}{\lap \phi}^\frac{6}{5} \normLp{2}{\delta^\epsilon}^2 + 
        \frac{1}{8}\normLp{2}{\lap \delta^\epsilon}^2, \\
        N_3 &= 
        \abs{\lp
            \bee(\grad\phi) \cdot \grad \delta^\epsilon, \delta^\epsilon
        \rp} \leq 
        C\normLp{2}{\phi}^\frac{1}{4}\normLp{2}{\lap\phi}^\frac{3}{4}
        \normLp{2}{\delta^\epsilon}^2 + 
        \frac{1}{16}\normLp{2}{\lap\delta^\epsilon}^2,
    \end{align*}
    and 
    \begin{align*}
    N_4 &= 
    \abs{\lp
        \lp \bee(\grad\phi) - \grad \phi \rp \cdot \grad \phi,
        \delta^\epsilon
    \rp } \\ & \leq
    C \epsilon^\alpha \int_{\nT^2} \abs{\grad \phi}^{\beta+1}\abs{\delta^\epsilon}d\bx \\ &\leq 
    C \epsilon^\alpha \normLp{2\beta+2}{\grad \phi}^{\beta+1}  \normLp{2}{\delta^\epsilon}.
    \end{align*}
    Applying the Sobolev inequality, we deduce that  
    \[
    \normLp{2\beta+2}{\grad\phi}^{\beta+1} \leq 
    C \normLp{2}{\grad\phi} \normHs{1}{\grad\phi}^\beta
    \leq C \normLp{2}{\phi}^\frac{1}{2}\normHs{2}{\phi}^{\beta + \frac{1}{2}}.
    \]
    Inserting our bounds for each $N_i$ into \eqref{cKSE_scal_conv1} and rearranging then yields
    \begin{align}\label{cKSE_scal_conv2}
        \frac{d}{dt}\normLp{2}{\delta^\epsilon}^2 + \frac{1}{16}\normLp{2}{\lap \delta^\epsilon}^2  &\leq 
        C \normLp{2}{\delta^\epsilon}^6 + 
        \normLp{2}{\delta^\epsilon}^2 +
        C \epsilon^\alpha 
        \normLp{2}{\phi}^\frac{1}{2}\normHs{2}{\phi}^{\beta + \frac{1}{2}}
        \normLp{2}{\delta^\epsilon}  \\ & \quad + 
        C\lp 
        \normLp{2}{\phi}^\frac{2}{5}\normLp{2}{\lap\phi}^\frac{6}{5} + 
        \normLp{2}{\phi}^\frac{1}{4}\normLp{2}{\lap\phi}^\frac{3}{4} 
        \rp\normLp{2}{\delta^\epsilon}^2. \notag
    \end{align}
    Now we apply the ansatz 
    \[ 
    \normLp{2}{\delta^\epsilon} < 1
    \]
    to obtain the bound 
    \[
    \normLp{2}{\delta^\epsilon}^6 \leq \normLp{2}{\delta^\epsilon}^2.
    \]
    We apply this estimate to \eqref{cKSE_scal_conv2} and eliminate $\normLp{2}{\delta^\epsilon}$ from each term to obtain 
    \begin{align}\label{cKSE_scal_conv3}
        \frac{d}{dt}\normLp{2}{\delta^\epsilon} &\leq 
        C \normLp{2}{\phi}^\frac{1}{2}
        \normHs{2}{\phi}^{\beta + \frac{1}{2}}
        \epsilon^\alpha 
        \\&\quad  + 
        C\lp 
        1 +
        \normLp{2}{\phi}^\frac{2}{5}\normLp{2}{\lap\phi}^\frac{6}{5} + 
        \normLp{2}{\phi}^\frac{1}{4}\normLp{2}{\lap\phi}^\frac{3}{4} 
        \rp\normLp{2}{\delta^\epsilon}.
        \notag
    \end{align}
    The term 
    \[
    1 +
        \normLp{2}{\phi}^\frac{2}{5}\normLp{2}{\lap\phi}^\frac{6}{5} + 
        \normLp{2}{\phi}^\frac{1}{4}\normLp{2}{\lap\phi}^\frac{3}{4} 
    \]
    is always integrable and the term
    \[
    \normLp{2}{\phi}^\frac{1}{2}
    \normHs{2}{\phi}^{\beta + \frac{1}{2}}
    \]
    is integrable for $\beta \in [0, \frac{3}{2}]$. It now follows from Gr\"onwall's inequality that 
    \begin{align}
        \normLp{2}{\delta^\epsilon(t)} \leq
        e^{A(t)}\normLp{2}{\delta^\epsilon(0)} + 
        e^{A(t)}B(t)\epsilon^\alpha \leq e^{A(T)}B(T)\epsilon^\alpha,
    \end{align}
    using the fact that $\delta^\epsilon(0) = 0$, and with 
    \begin{align*}
        A(t) &= C\int_0^t
        1 +
        \normLp{2}{\phi}^\frac{2}{5}\normLp{2}{\lap\phi}^\frac{6}{5} + 
        \normLp{2}{\phi}^\frac{1}{4}\normLp{2}{\lap\phi}^\frac{3}{4} ds, \\
        B(t) &= C\int_0^t \normLp{2}{\phi}^\frac{1}{2}
        \normHs{2}{\phi}^{\beta + \frac{1}{2}} ds. 
    \end{align*}
    By taking $\epsilon$ sufficiently small, we have for all $0\leq t \leq T$
    \[
    \normLp{2}{\delta^\epsilon(t)} < 1.
    \]
    It follows from a bootstrapping argument that 
    \begin{align}
        \|\delta^\epsilon(t)\|_{L^\infty(0,T;L^2)} \leq e^{A(T)}B(T)\epsilon^\alpha.
    \label{cKSE_scal_conv_LinfL2}
    \end{align}
    Now we integrate \eqref{cKSE_scal_conv2} on $[0,T]$, again using that $\normLp{2}{\delta^\epsilon}^6 \leq \normLp{2}{\delta^\epsilon}^2$, and apply to obtain 
    \begin{align}
        \int_0^T\normLp{2}{\lap\delta^\epsilon}^2 dt &\leq 
        C\epsilon^\alpha B(T) e^{A(T)}B(T)\epsilon^\alpha + 
        A(T)\lp e^{A(T)}B(T)\epsilon^\alpha \rp^2 
         \\ &\leq 
        K(T)^2\epsilon^{2\alpha}, \notag
    \end{align}
    where 
    \[
    K(T)^2 = CB(T)^2e^{A(T)} + A(T)B(T)e^{A(T)}.
    \]
    Therefore we obtain 
    \begin{align}
        \|\delta^\epsilon \|_{L^2(0,T;H^2)} \leq 
        \lp Te^{A(T)}B(T) + K(T)\rp \epsilon^\alpha.
    \end{align}
\end{proof}

 \section{Computational Results}\label{sec_Computational_Results}
In this section, we examine the calmed Kuramoto-Sivashinsky equations computationally via several simulations, where the calming function $\bee = \bee_i$ is described in $\eqref{eta_choices}$.  We include snapshots of the evolution of solutions for the different choices of $\bee$ in Figure \ref{fig:Dynamics_AllTypes}, and for different choices of $\epsilon$ in Figure \ref{fig:Dynamics_type3} (we show results for $\bee_3$ only for the sake of brevity; $\bee_1$ and $\bee_2$ yielded qualitatively similar results). The former illustrates the different effects of the choice of $\bee$ on the dynamics, while the latter indicates the uniform convergence of $\bu^\epsilon$ to $\bu$.

In addition, we examine convergence rates in $L^\infty(0,T;L^2)$, $L^\infty(0,T;L^\infty)$, and $L^2(0,T;H^2)$ for $\bee_1$ (Figure \ref{fig:ic1t1}), $\bee_2$ (Figure \ref{fig:ic1t2}), and $\bee_3$ (Figure \ref{fig:ic1t3}) with initial data \eqref{IniData0} as $\epsilon\rightarrow0^+$ (for simplicity, we set $T=1$, since with all our initial data, solutions to KSE appear to be quite stable on $[0,1]$).  We find that the powers on the $L^\infty(0,T;L^2)$ and $L^2(0,T; H^2)$ convergence rates in Corollary \ref{conv_particular} appear to be sharp. 

Finally, in Figures \ref{fig:ic2t1}, \ref{fig:ic2t2}, and \ref{fig:ic2t3} we check the robustness of the convergence with respect to larger initial data \eqref{IniData1} for $\bee_1$, $\bee_2$, and $\bee_3$. In comparing initial data \eqref{IniData0} with \eqref{IniData1}, we find very little qualitative variation in the error rates, indicating that changes in initial data will only marginally change the error between solutions to KSE and solutions to calmed KSE for $\epsilon > 0$ sufficiently small.

\subsection{Numerical Methods}
    All computations were done in Matlab (R$2021$a) using pseudo-spectral methods with the standard $2/3's$ dealiasing for the nonlinear term. To evolve the system, we used a well-known modification of the Runge-Kutta-4 time-stepping scheme adapted to handle the linear terms implicitly via an integrating factor to handle the nonlinear terms implicitly (see, e.g., \cite{Kassam_Trefethen_2005}) 
    with time step $\Delta t\approx 4.2943\times10^{-4}$ chosen to respect the maximum advective CFL condition in Figures \ref{fig:Dynamics_AllTypes}, \ref{fig:Dynamics_type3}, \ref{fig:ic1t1}, \ref{fig:ic1t2}, and \ref{fig:ic1t3}, with later figures having a rescaled time step $\Delta t = 1.0736\times10^{-4}$. Our simulations for KSE and cKSE were resolved\footnote{Note: For the Kuramoto-Sivashinsky equations (calmed or otherwise), even in fairly chaotic regimes, one often does not need especially high resolution, due to the strong hyperdiffusion term.  Moreover, so long as the solution is  well-resolved, which we take to mean that the energy spectrum at the modes higher than the 2/3's dealiasing cut-off is at or below machine precision (roughly $2.22\times10^{-16}$), increasing the resolution only increases round-off error, due to the additional computations being performed.  Hence, to minimize roundoff error, we purposely chose the fairly low resolution of $128^2$, although our higher-resolution tests, not reported here, produced qualitatively similar results.} with a spatial mesh of $128^2$. All computations were done using the nondimensionalized calmed Kuramoto-Sivashinsky equations, 
    \begin{subequations}
    \begin{align}
    \dt{\bu}+(\boldeta^\epsilon(\bu)\cdot\nabla)\bu + \lambda \triangle\bu + \triangle^2\bu &= \mathbf{0},
    \\
    \bu(\bx,0)&=\bu_0(\bx),
    \end{align}
    \end{subequations}
    over the periodic domain $\Omega = [-\pi, \pi)^2$ for $\lambda > 0$. \\
    Throughout this section, a \emph{type} $1$, \emph{type} $2$, or \emph{type} $3$ solution is a solution to calmed KSE with calming function $\bee_1$, $\bee_2$, or $\bee_3$ respectively. 
\subsection{Simulations}

Here, we take initial conditions to be 
\begin{align} \label{IniData0}
    \bu_0(x,y) = \binom{\cos(x+y) + \cos(x)}{ \cos(x+y) + \cos(y) }
\end{align}
and all color plots seen below are plots of the magnitude $|\bu|=\abs{\lp u, v \rp} = \sqrt{u^2 + v^2}$.
In all plots of solutions, the horizontal axis corresponds to the $y$-axis and the vertical axis corresponds to the $x$-axis.

Our choice for initial data $\bu_0$ was motivated by the choice of scalar initial data found in \cite{Kalogirou_Keaveny_Papageorgiou_2015},  \cite{Larios_Yamazaki_2020_rKSE}, and \cite{Larios_Rahman_Yamazaki_2021_JNLS_KSE_PS}; namely,
\[\phi_0(x,y) = \sin(x+y) + \sin(x) + \sin(y).\] 
Hence, we set $\bu_0 = \grad \phi_0$.



\begin{figure}[ht]
\centering
    \begin{subfigure}{.242\textwidth}
          \centering
          \includegraphics[width=1\linewidth,trim = 47mm 21mm 42mm 17mm, clip]{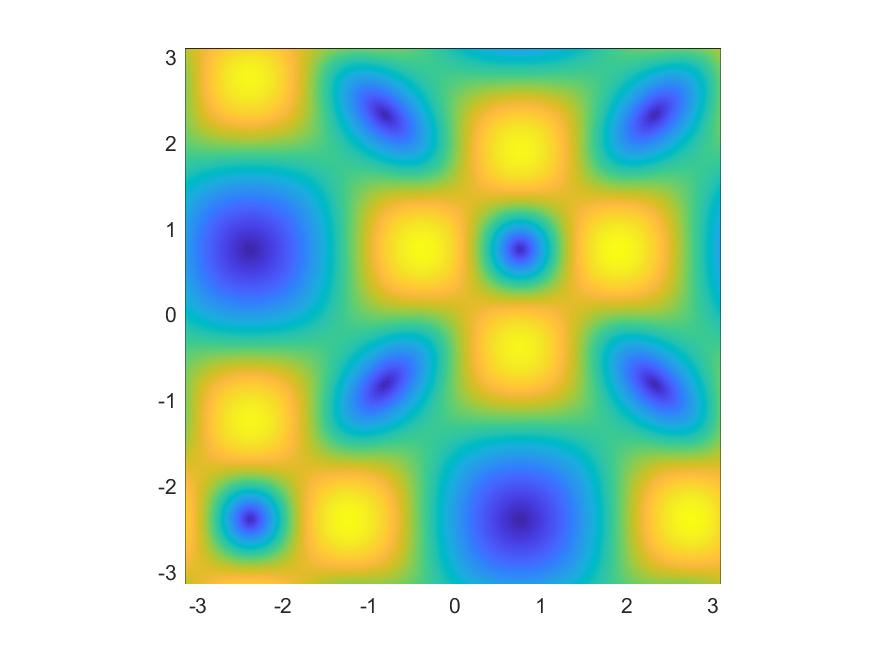} 
          \caption{Type $1$}
          \label{fig:DynType1}
    \end{subfigure}
    \begin{subfigure}{.242\textwidth}
          \centering
          \includegraphics[width=1\linewidth,trim = 47mm 21mm 42mm 17mm, clip]{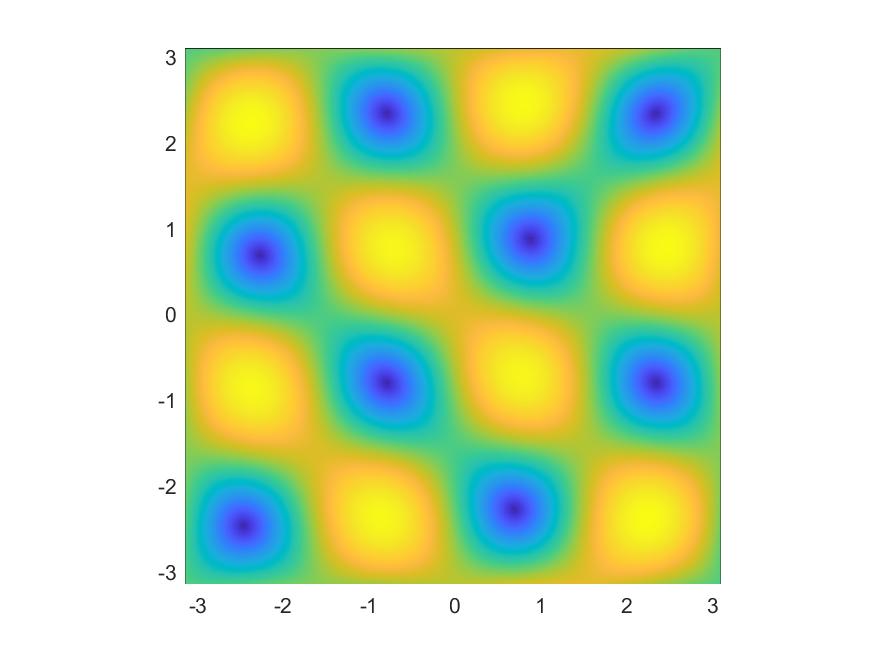}
          \caption{Type $2$}
          \label{fig:DynType2}
    \end{subfigure}
    \begin{subfigure}{.242\textwidth}
          \centering
          \includegraphics[width=1\linewidth,trim = 47mm 21mm 42mm 17mm, clip]{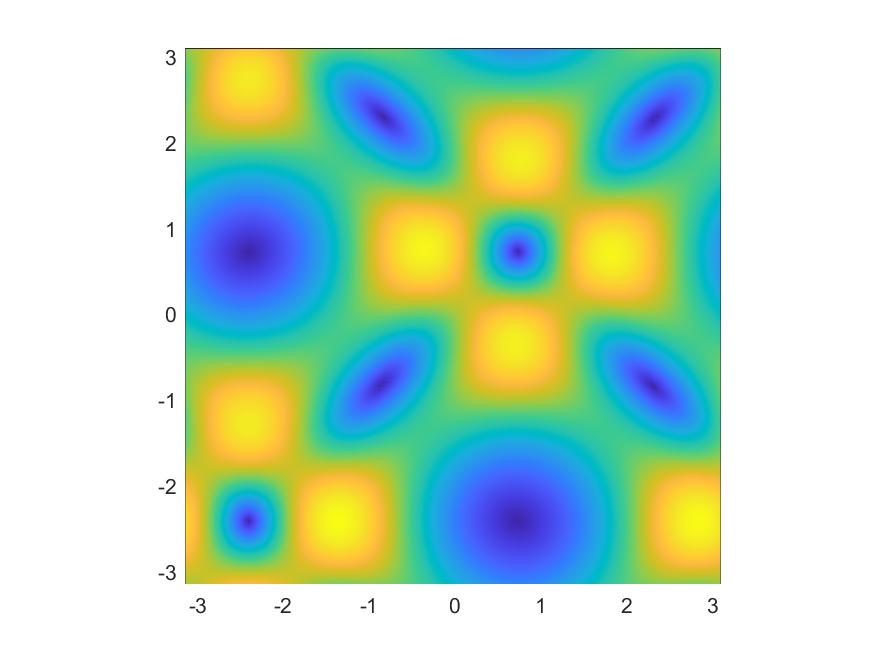}
          \caption{Type $3$}
          \label{fig:DynType3}
    \end{subfigure}
    \begin{subfigure}{.242\textwidth}
          \centering
          \includegraphics[width=1\linewidth,trim = 47mm 21mm 42mm 17mm, clip]{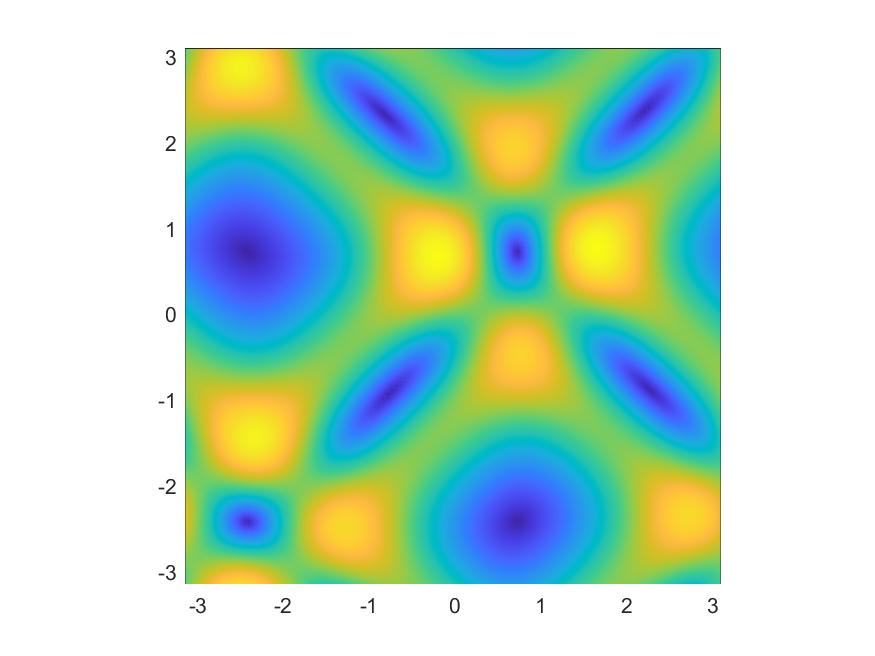}
          \caption{KSE}
          \label{fig:DynKSE}
    \end{subfigure}
    \caption{
    Solutions to calmed KSE of each type compared with a solution to KSE at time $t = 2$, with $\epsilon = 0.1$, $\lambda = 4.1$, and $\bu_0$ given by \eqref{IniData0}. 
    }\label{fig:Dynamics_AllTypes} 
\end{figure}
\FloatBarrier

Though some differences can be seen among the images above, one can see that each type of calmed KSE solution approximates the overall behavior of a KSE solution. One can also observe that the accuracy of the approximation varies by type. 

\begin{figure}[ht]
\centering
    \begin{subfigure}{.242\textwidth}
          \centering
          \includegraphics[width=1.0\linewidth,trim = 47mm 21mm 42mm 17mm, clip]{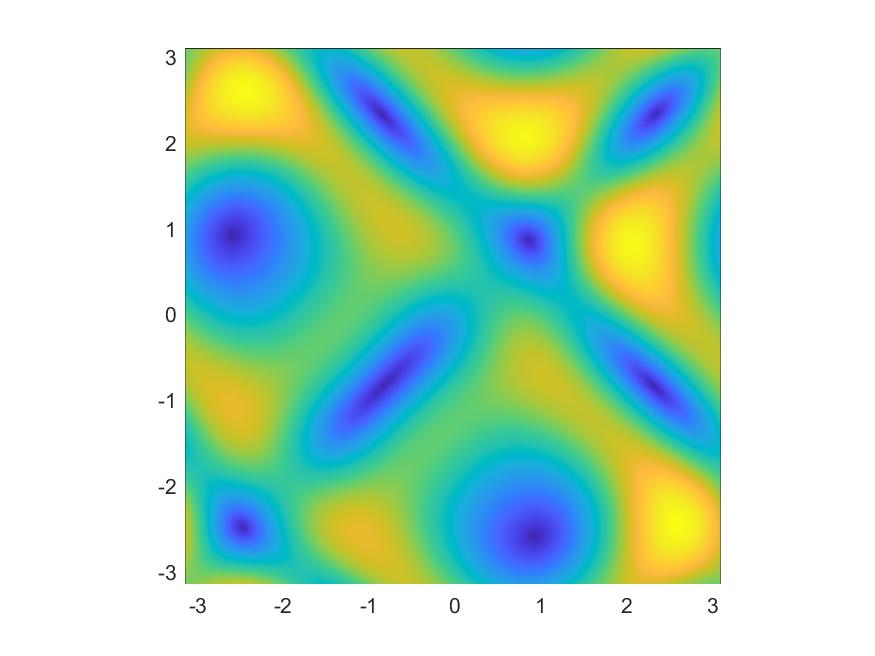} 
    \end{subfigure}
    \begin{subfigure}{.242\textwidth}
          \centering
          \includegraphics[width=1.0\linewidth,trim = 47mm 21mm 42mm 17mm, clip]{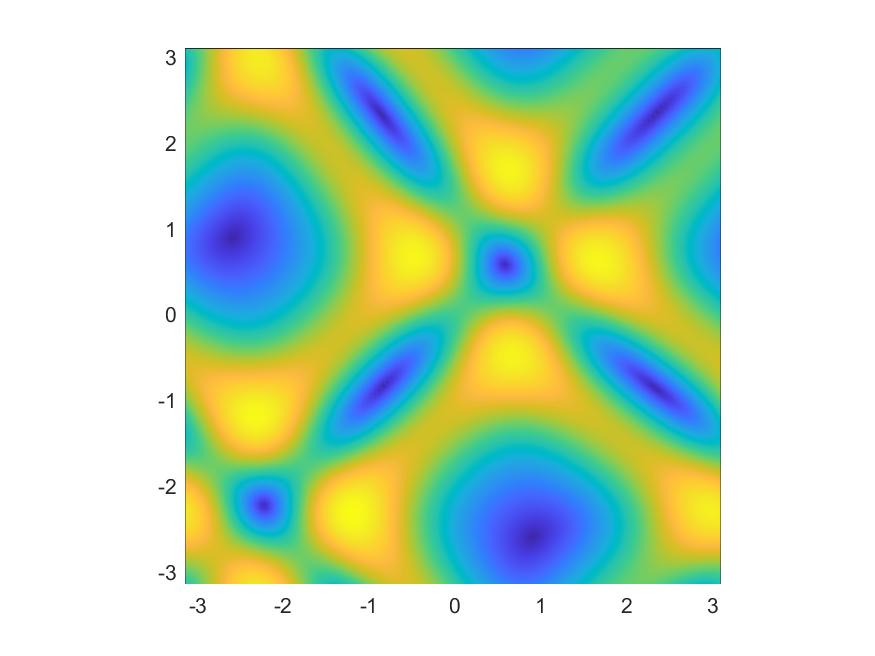}
    \end{subfigure}
    \begin{subfigure}{.242\textwidth}
          \centering
          \includegraphics[width=1.0\linewidth,trim = 47mm 21mm 42mm 17mm, clip]{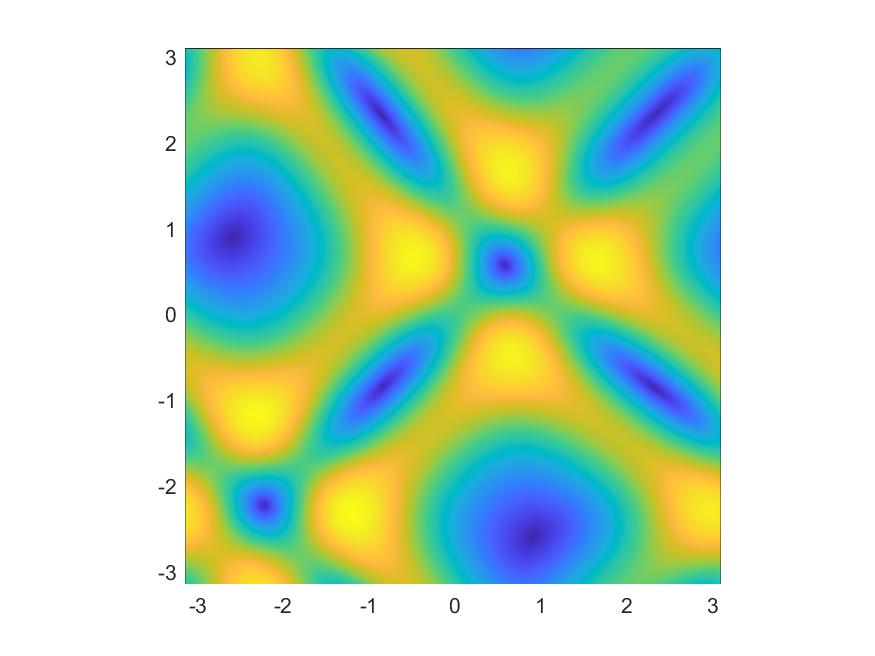}
    \end{subfigure}
    \begin{subfigure}{.242\textwidth}
          \centering
          \includegraphics[width=1.0\linewidth,trim = 47mm 21mm 42mm 17mm, clip]{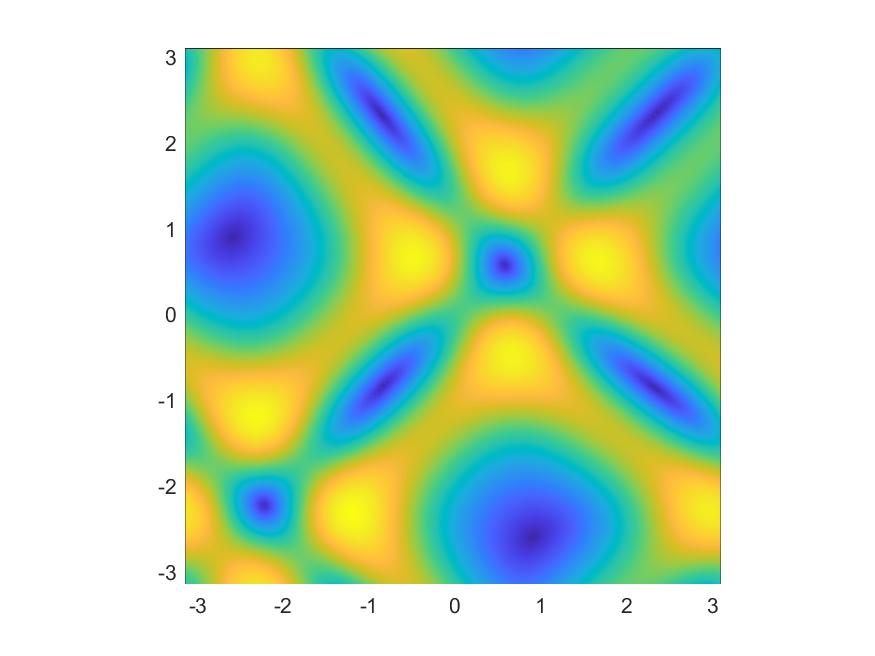}
    \end{subfigure}
    \begin{subfigure}{.242\textwidth}
          \centering
          \includegraphics[width=1.0\linewidth,trim = 47mm 21mm 42mm 17mm, clip]{KSE_lam_4p1_eps_0p1_N_128_type_3_t_2p00}
    \end{subfigure}
    \begin{subfigure}{.242\textwidth}
          \centering
          \includegraphics[width=1.0\linewidth,trim = 47mm 21mm 42mm 17mm, clip]{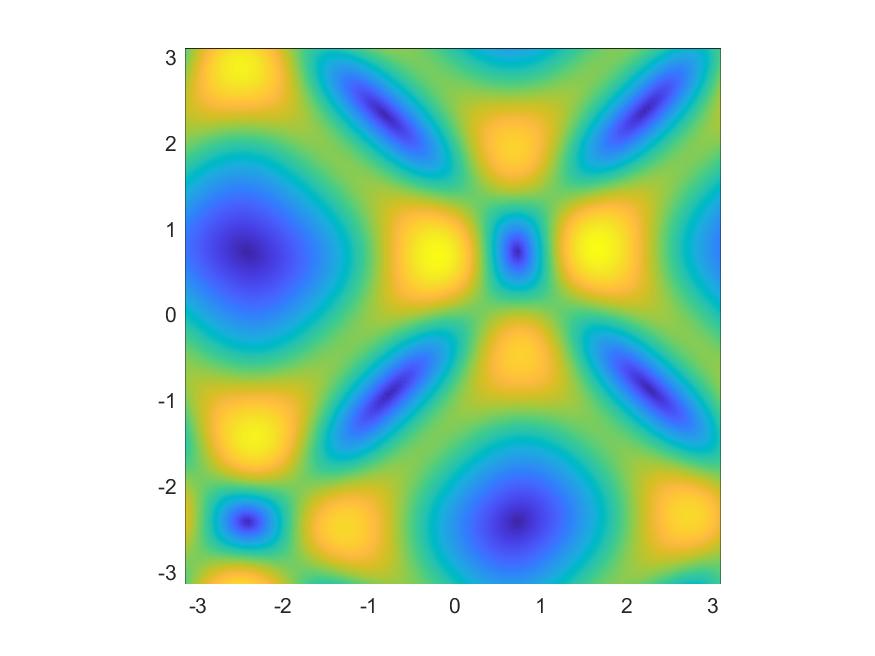}
    \end{subfigure}
    \begin{subfigure}{.242\textwidth}
          \centering
          \includegraphics[width=1.0\linewidth,trim = 47mm 21mm 42mm 17mm, clip]{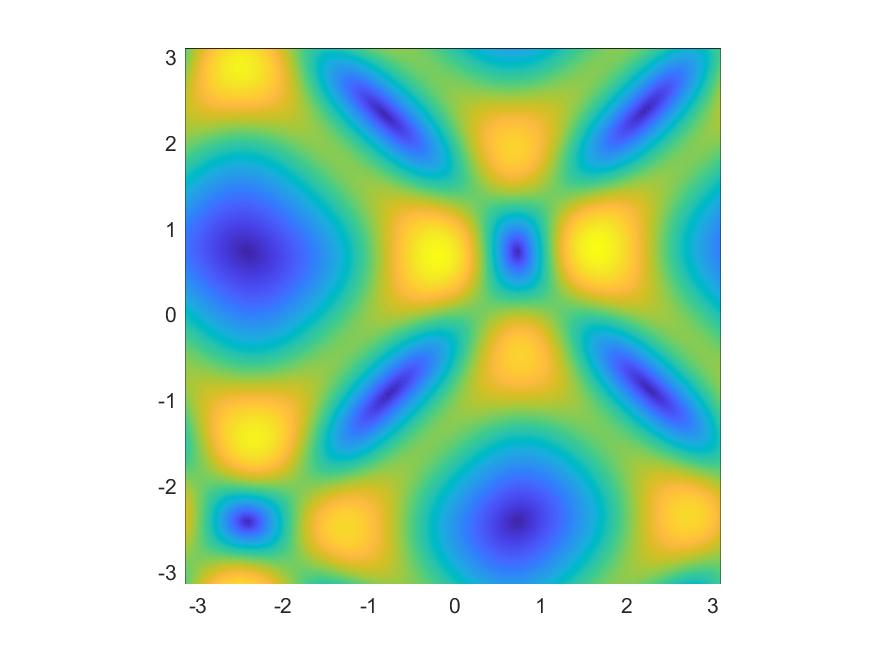}
    \end{subfigure}
    \begin{subfigure}{.242\textwidth}
          \centering
          \includegraphics[width=1.0\linewidth,trim = 47mm 21mm 42mm 17mm, clip]{KSE_lam_4p1_eps_0p1_N_128_type_0_t_2p00}
    \end{subfigure}
    \begin{subfigure}{.242\textwidth}
          \centering
          \includegraphics[width=1.0\linewidth,trim = 47mm 21mm 42mm 17mm, clip]{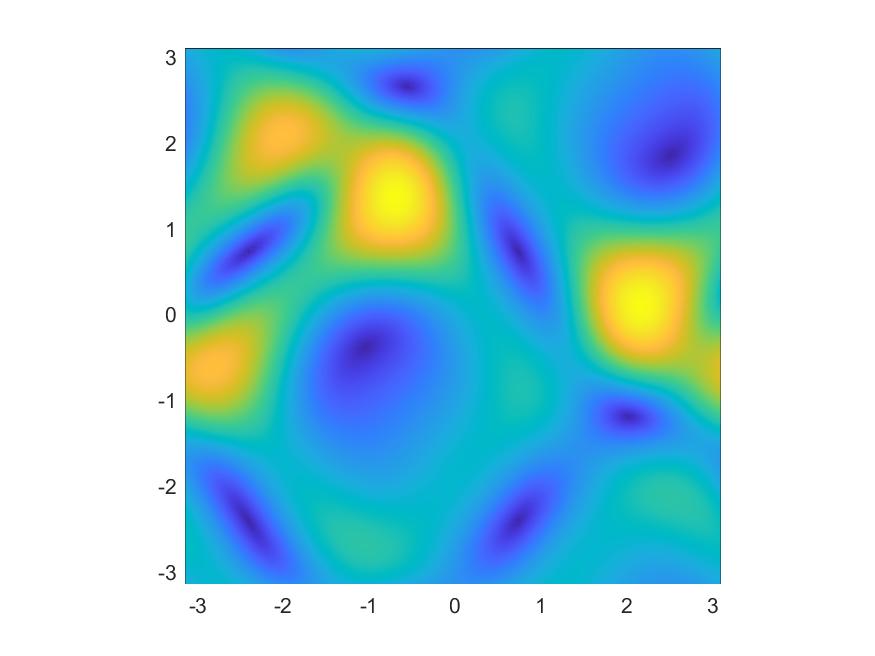} 
    \end{subfigure}
    \begin{subfigure}{.242\textwidth}
          \centering
          \includegraphics[width=1.0\linewidth,trim = 47mm 21mm 42mm 17mm, clip]{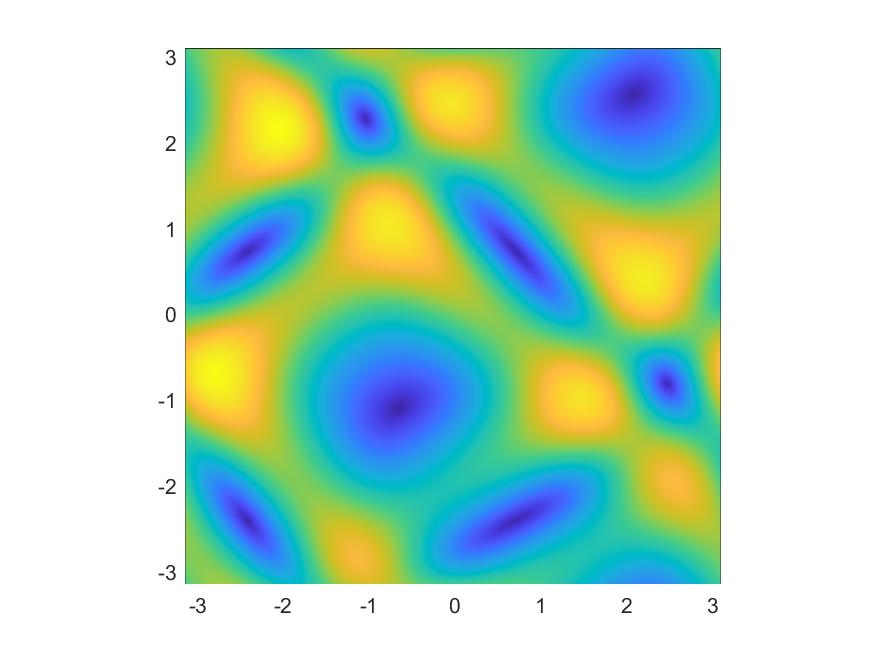}
    \end{subfigure}
    \begin{subfigure}{.242\textwidth}
          \centering
          \includegraphics[width=1.0\linewidth,trim = 47mm 21mm 42mm 17mm, clip]{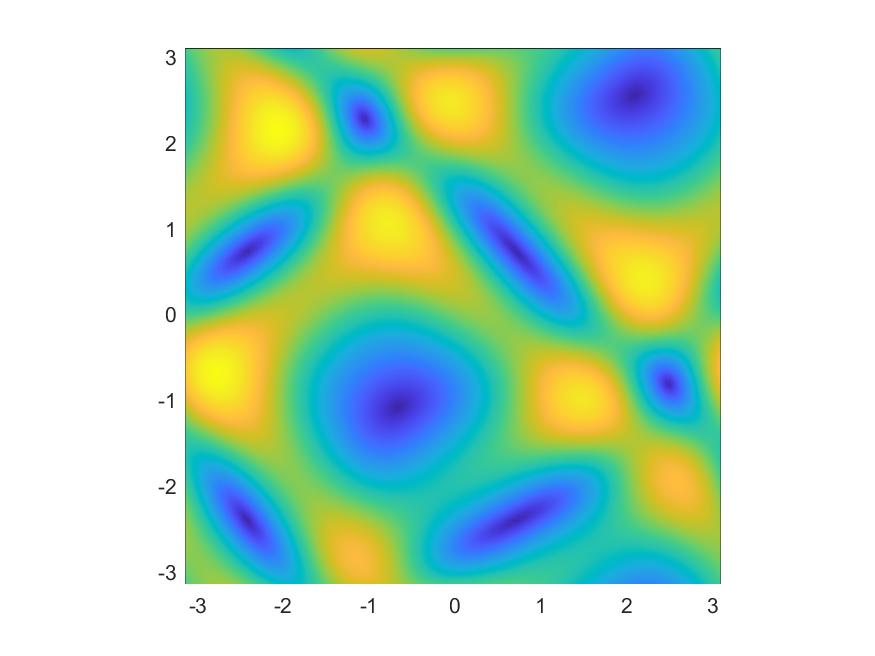}
    \end{subfigure}
    \begin{subfigure}{.242\textwidth}
          \centering
          \includegraphics[width=1.0\linewidth,trim = 47mm 21mm 42mm 17mm, clip]{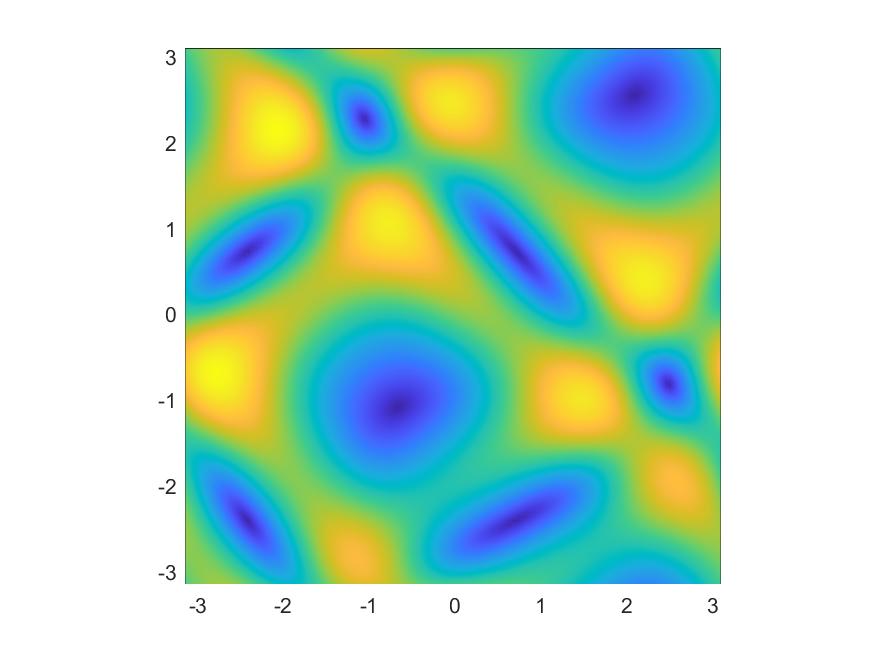}
    \end{subfigure}
    \begin{subfigure}{.242\textwidth}
          \centering
          \includegraphics[width=1.0\linewidth,trim = 47mm 21mm 42mm 17mm, clip]{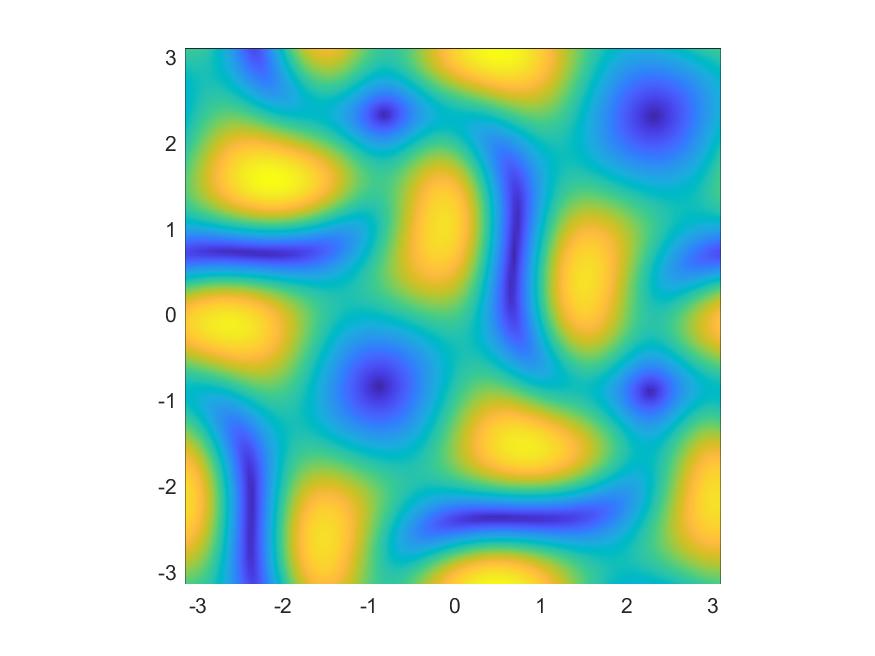} 
    \end{subfigure}
    \begin{subfigure}{.242\textwidth}
          \centering
          \includegraphics[width=1.0\linewidth,trim = 47mm 21mm 42mm 17mm, clip]{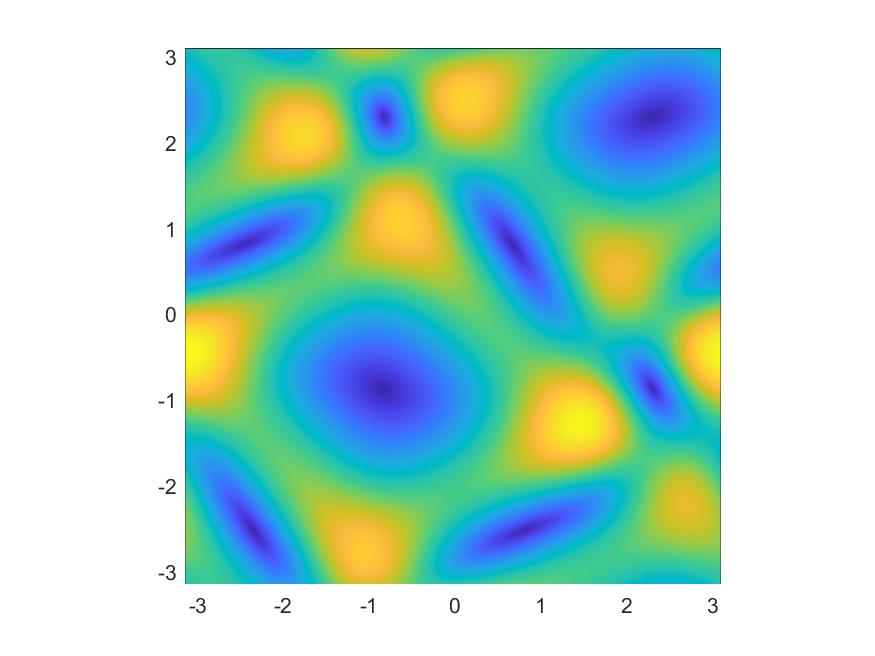}
    \end{subfigure}
    \begin{subfigure}{.242\textwidth}
          \centering
          \includegraphics[width=1.0\linewidth,trim = 47mm 21mm 42mm 17mm, clip]{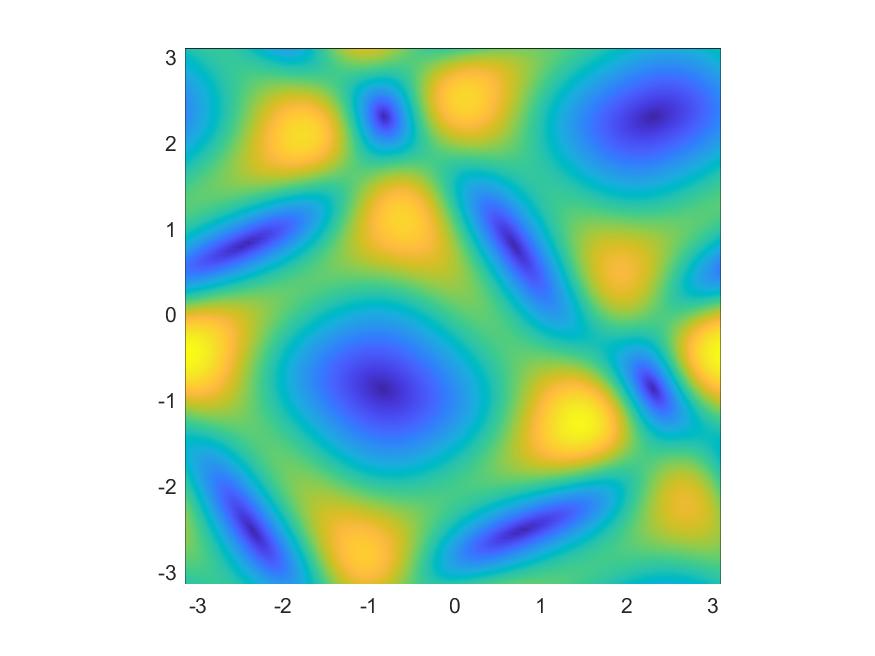}
    \end{subfigure}
    \begin{subfigure}{.242\textwidth}
          \centering
          \includegraphics[width=1.0\linewidth,trim = 47mm 21mm 42mm 17mm, clip]{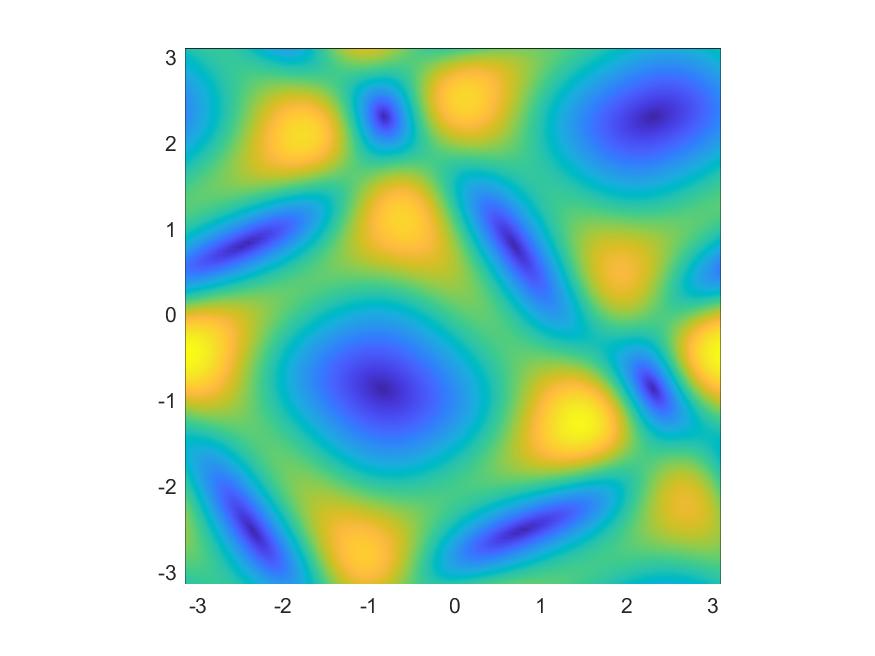}
    \end{subfigure}
\begin{subfigure}{.242\textwidth}
      \centering
      \includegraphics[width=1.0\linewidth,trim = 47mm 21mm 42mm 17mm, clip]{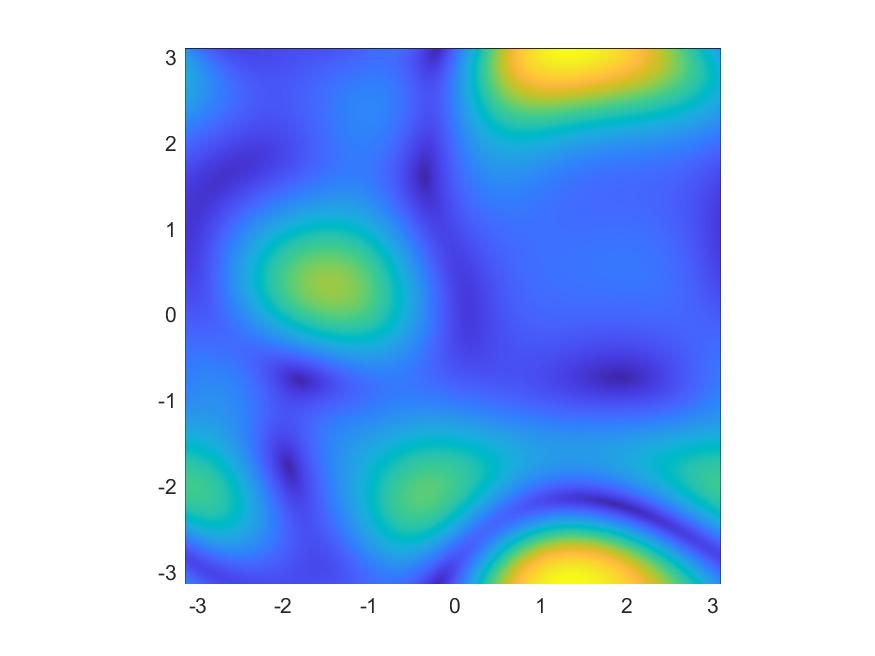} 
      \caption{$\epsilon_1 = 0.1$}
      \label{fig:cf1_eps1}
\end{subfigure}
\begin{subfigure}{.242\textwidth}
      \centering
      \includegraphics[width=1.0\linewidth,trim = 47mm 21mm 42mm 17mm, clip]{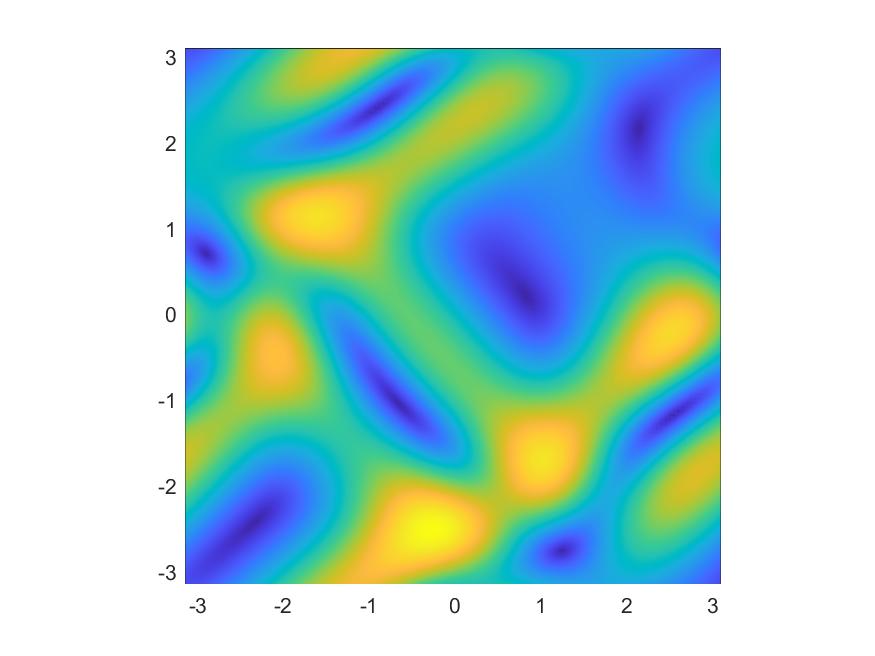}
      \caption{$\epsilon_2 = 0.01$}
      \label{fig:cf1_eps2}
\end{subfigure}
\begin{subfigure}{.242\textwidth}
      \centering
      \includegraphics[width=1.0\linewidth,trim = 47mm 21mm 42mm 17mm, clip]{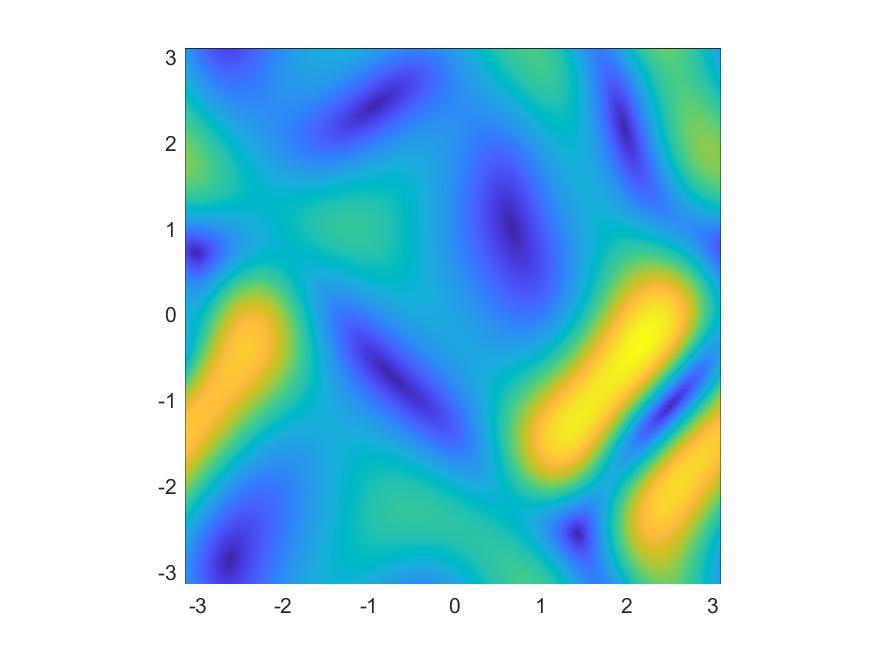}
      \caption{$\epsilon_3 = 0.001$}
      \label{fig:cf1_eps3}
\end{subfigure}
\begin{subfigure}{.242\textwidth}
      \centering
      \includegraphics[width=1.0\linewidth,trim = 47mm 21mm 42mm 17mm, clip]{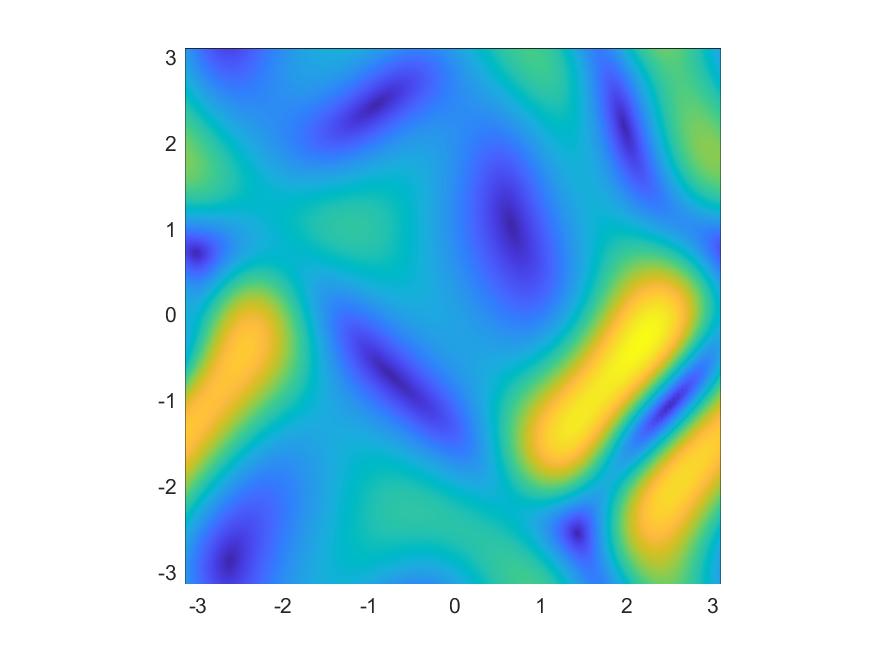}
      \caption{KSE}
      \label{fig:cf1_KSE}
\end{subfigure}
    \caption{
    \label{fig:Dynamics_type3} Column \ref{fig:cf1_KSE} is a solution to KSE \eqref{KSE} for $t= 1, \ldots, 5$, whereas columns \ref{fig:cf1_eps1}, \ref{fig:cf1_eps2}, and \ref{fig:cf1_eps3} are type 3 solutions to calmed KSE \eqref{cKSE} on the same time interval with $\epsilon \in \{0.1, 0.01, 0.001 \}$. In this figure, $\lambda = 4.1$ is fixed and initial data $\bu_0$ is given in \eqref{IniData0}. Viewing the pictures from left to right, we can see that $\bu^\epsilon \to \bu$ as $\epsilon\to 0$.} 
\end{figure}
\FloatBarrier
In Figure \ref{fig:Dynamics_type3} we focus only on type $3$ approximations to better illustrate how well calmed KSE solutions can approximate KSE solutions over time for various choices of $\epsilon$. Indeed, when viewed from left to right we can observe the convergence of our calmed KSE solutions to the original KSE solution.

\begin{figure}
    \centering
    \includegraphics[width=1.1\linewidth,trim = 40mm 14mm 1mm 20mm, clip]{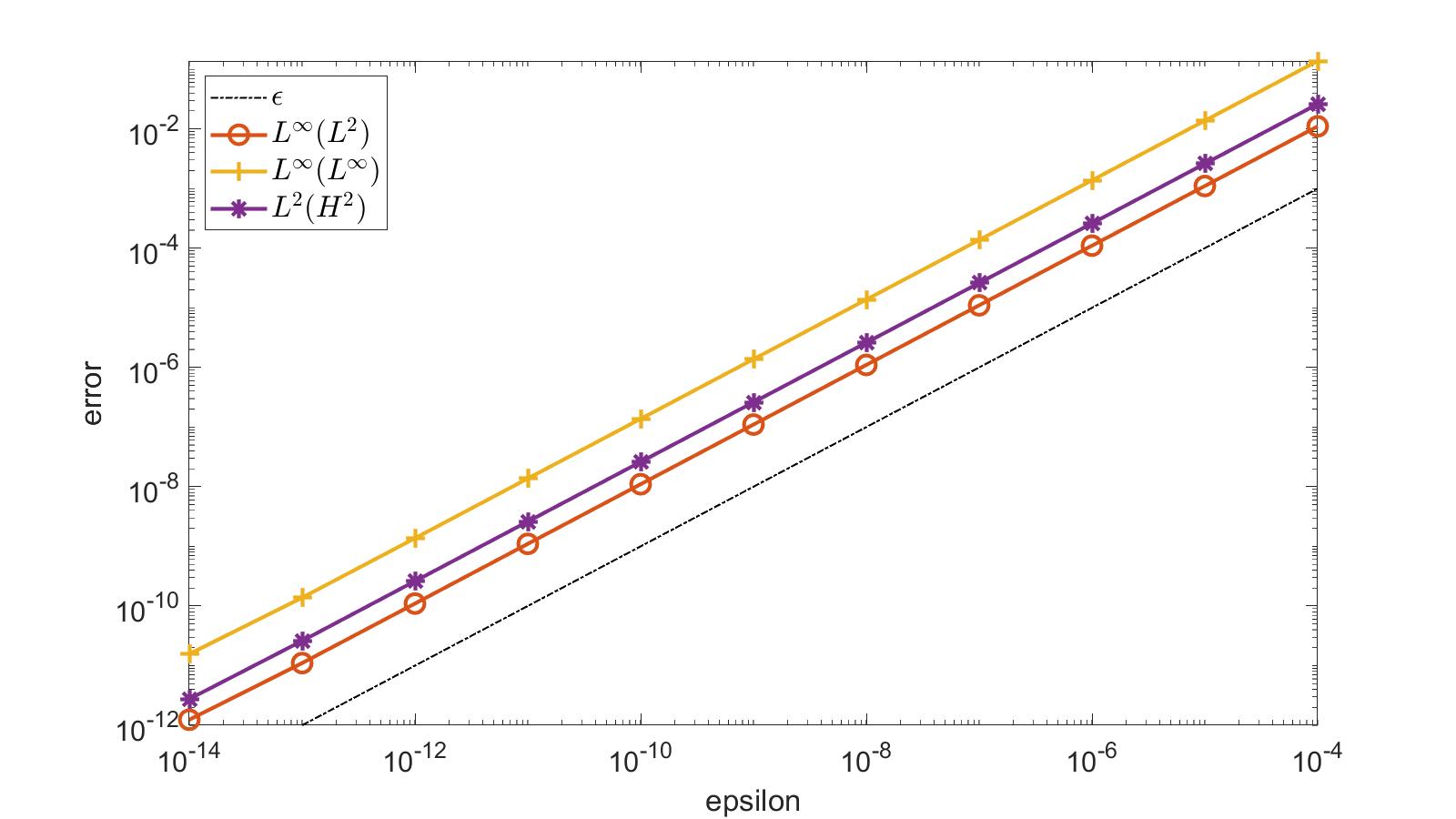} 
    \caption{Estimates of $\bu - \bu^\epsilon $ vs. $\epsilon$ in norms $\| \cdot \|_{L^\infty(0,T;L^2)}$, $\| \cdot \|_{L^\infty(0,T;L^\infty)}$, and $\| \cdot \|_{L^2(0,T;H^2)}$, at time $T=1$ with $\bu^\epsilon$ a type $1$ solution and with initial data given by \eqref{IniData0}. These estimates show a linear convergence rate.}
    \label{fig:ic1t1}
\end{figure}

\begin{figure}
    \centering
    \includegraphics[width=1.1\linewidth,trim = 40mm 14mm 1mm 20mm, clip]{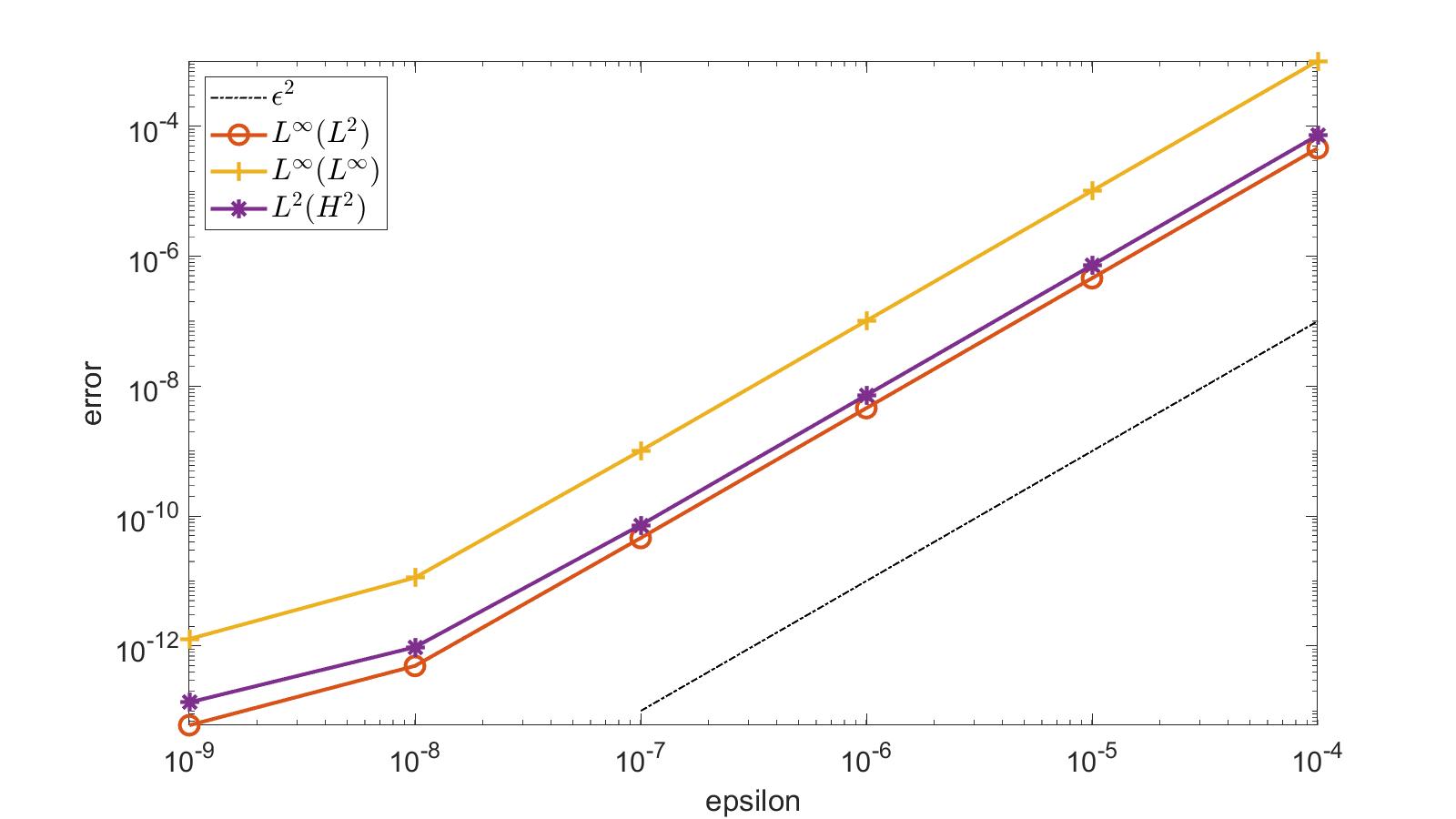} 
   \caption{Estimates of $\bu - \bu^\epsilon $ vs. $\epsilon$ in norms $\| \cdot \|_{L^\infty(0,T;L^2)}$, $\| \cdot \|_{L^\infty(0,T;L^\infty)}$, and $\| \cdot \|_{L^2(0,T;H^2)}$, at time $T=1$ with $\bu^\epsilon$ type a $2$ solution and with initial data given by \eqref{IniData0}. These estimates show a quadratic convergence rate.}
    \label{fig:ic1t2}
\end{figure}

\begin{figure}
    \centering
    \includegraphics[width=1.1\linewidth,trim = 40mm 14mm 1mm 20mm, clip]{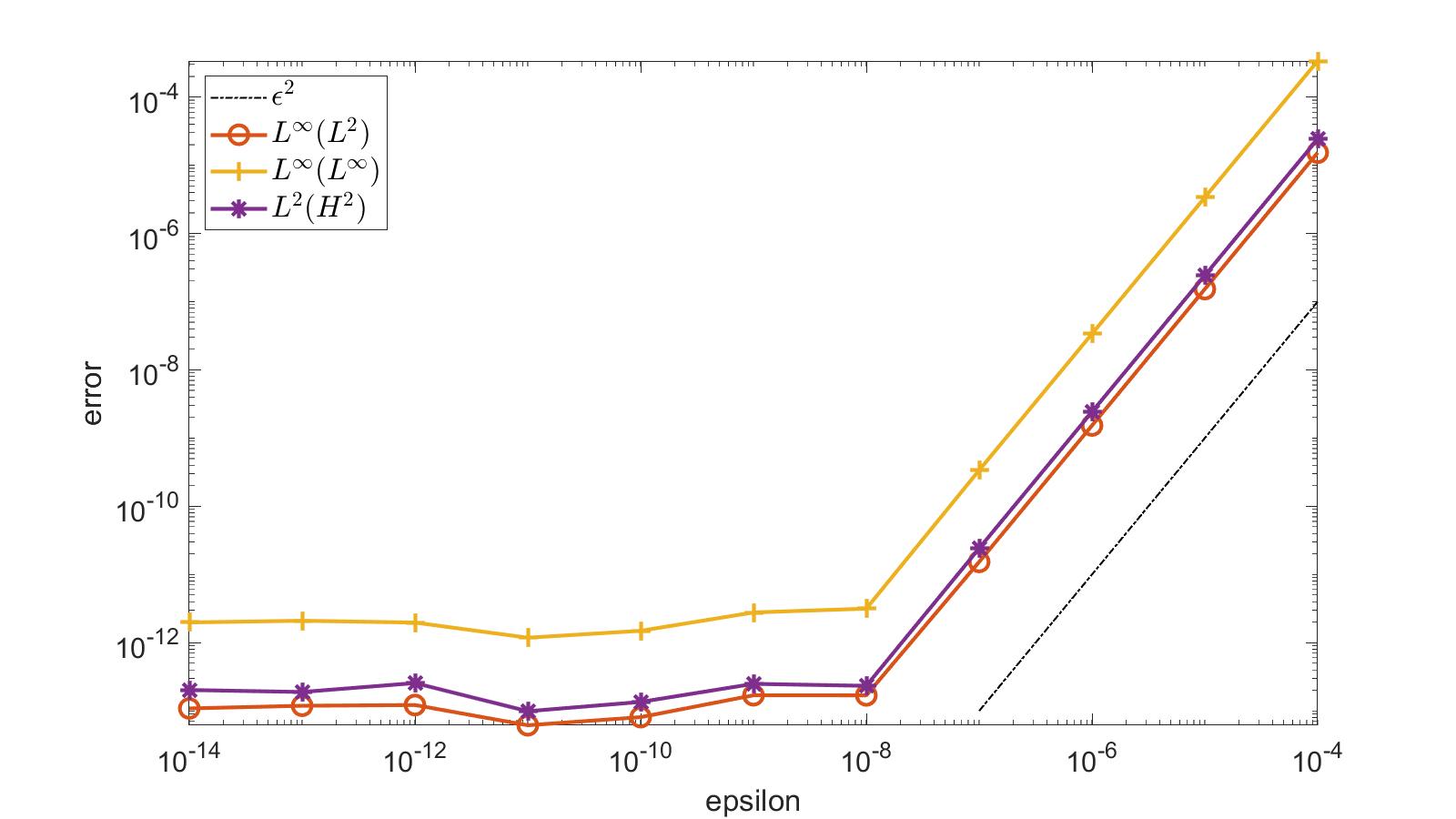} 
    \caption{Estimates of $\bu - \bu^\epsilon $ vs. $\epsilon$ in norms $\| \cdot \|_{L^\infty(0,T;L^2)}$, $\| \cdot \|_{L^\infty(0,T;L^\infty)}$, and $\| \cdot \|_{L^2(0,T;H^2)}$, at time $T=1$ with $\bu^\epsilon$ a type $3$ solution and with initial data given by \eqref{IniData0}. These estimates show a quadratic convergence rate.}
    \label{fig:ic1t3}
\end{figure}

In accordance with Corollary \ref{conv_particular} we see that solutions to calmed KSE corresponding to calming function $\bee_1$ yield a linear convergence rate whereas solutions to calmed KSE corresponding to calming functions $\bee_2$ or $\bee_3$ yield quadratic convergence rates. 

For additional testing, we choose initial data with higher oscillation and higher magnitude,
\begin{align} \label{IniData1}
\bu_0(x,y) = \binom{4 \lp \cos(x+y) + \sin(3x) \rp}
                   {4 \lp \cos(x+y) + \cos(4y) \rp},
\end{align}
and examine the convergence rates for each solution type. For each test of convergence, we fix $N=128^2$,  $T = 1$ and $\lambda = 4.1$. 


\begin{figure}
    \centering
    \includegraphics[width=1.1\linewidth,trim = 40mm 14mm 1mm 20mm, clip]{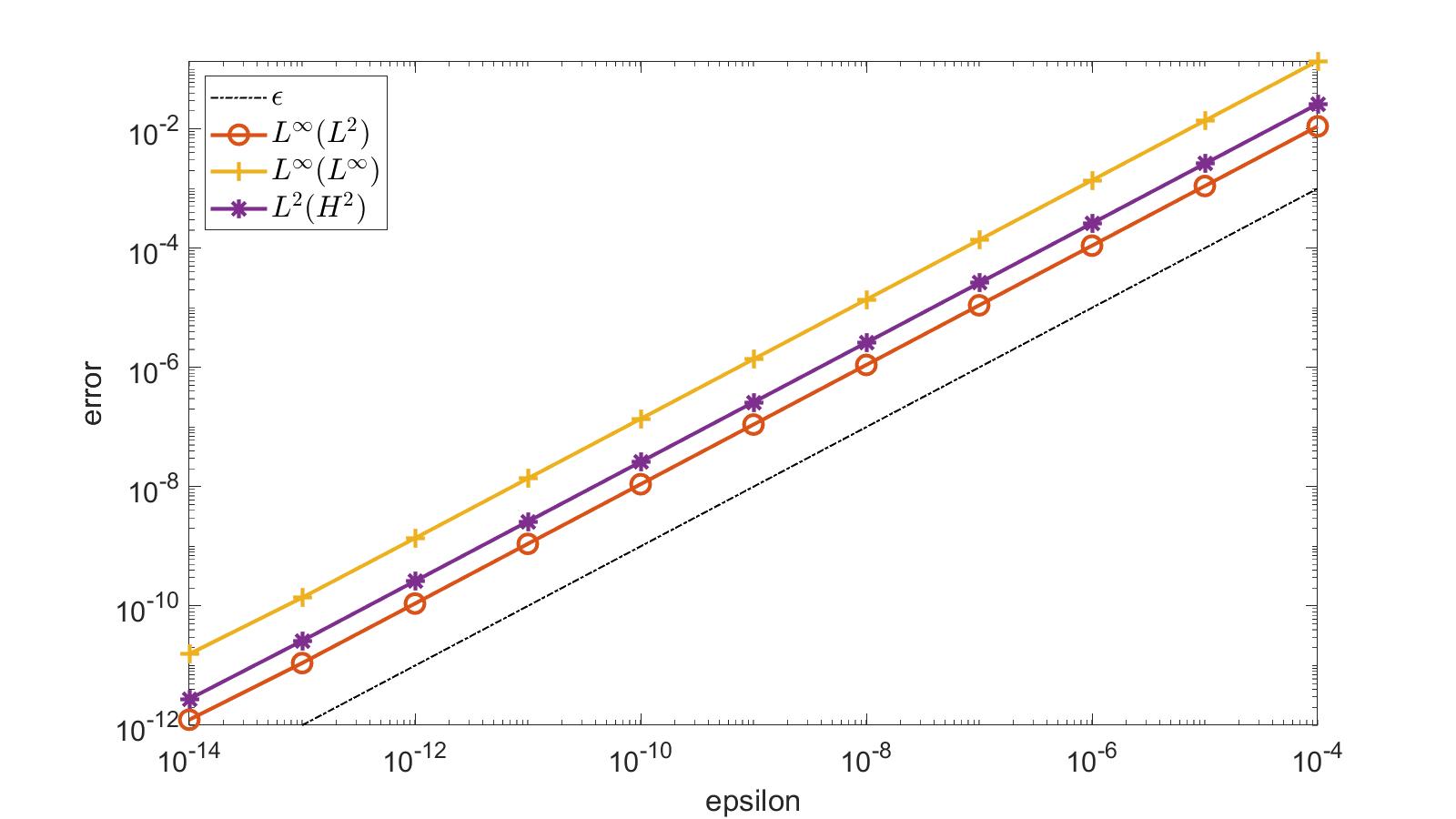} 
    \caption{Estimates of $\bu - \bu^\epsilon $ vs. $\epsilon$ in norms $\| \cdot \|_{L^\infty(0,T;L^2)}$, $\| \cdot \|_{L^\infty(0,T;L^\infty)}$, and $\| \cdot \|_{L^2(0,T;H^2)}$, at time $T=1$ with $\bu^\epsilon$ a type $1$ solution and with initial data given by \eqref{IniData1}. These estimates show a linear convergence rate.}
    \label{fig:ic2t1}
\end{figure}

\begin{figure}
    \centering
    \includegraphics[width=1.1\linewidth,trim = 40mm 14mm 1mm 20mm, clip]{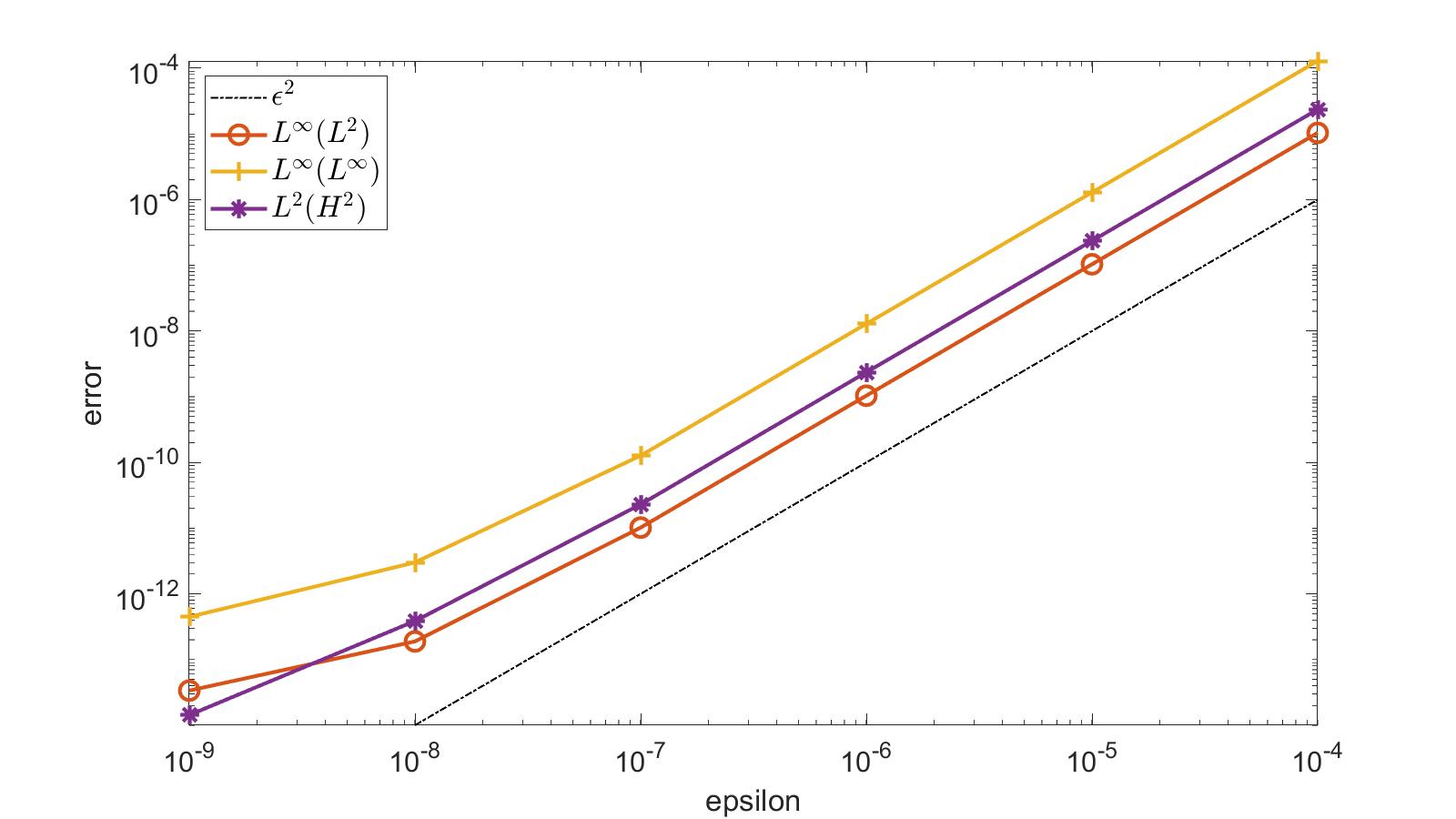} 
   \caption{Estimates of $\bu - \bu^\epsilon $ vs. $\epsilon$ in norms $\| \cdot \|_{L^\infty(0,T;L^2)}$, $\| \cdot \|_{L^\infty(0,T;L^\infty)}$, and $\| \cdot \|_{L^2(0,T;H^2)}$, at time $T=1$ with $\bu^\epsilon$ type a $2$ solution and with initial data given by \eqref{IniData1}. These estimates show a quadratic convergence rate.}
    \label{fig:ic2t2}
\end{figure}

\begin{figure}
    \centering
    \includegraphics[width=1.1\linewidth,trim = 40mm 14mm 1mm 20mm, clip]{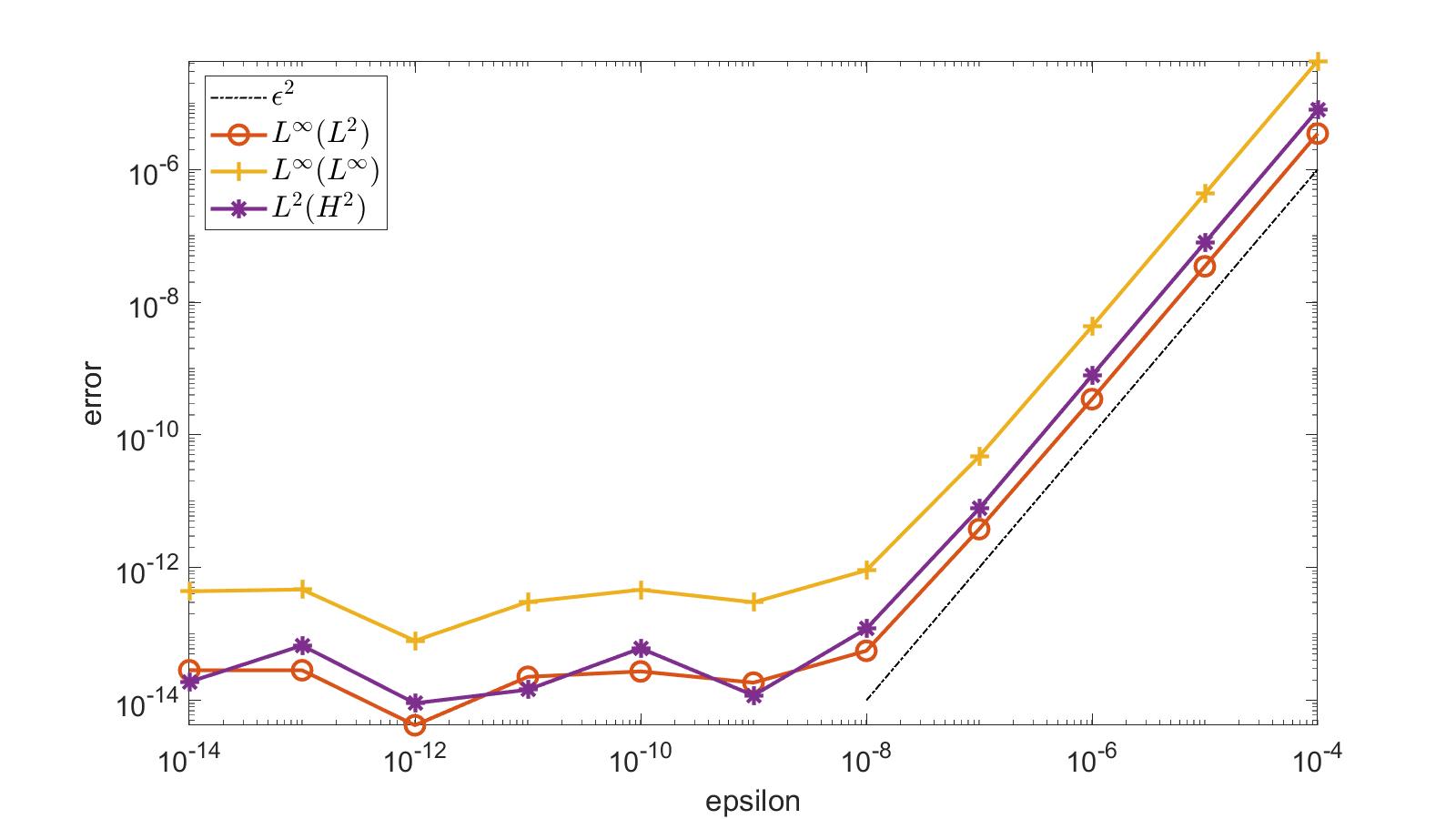} 
    \caption{Estimates of $\bu - \bu^\epsilon $ vs. $\epsilon$ in norms \newline $\| \cdot \|_{L^\infty(0,T;L^2)}$, $\| \cdot \|_{L^\infty(0,T;L^\infty)}$, and $\| \cdot \|_{L^2(0,T;H^2)}$, at time $T=1$ with $\bu^\epsilon$ a type $3$ solution and with initial data given by \eqref{IniData1}. These estimates show a quadratic convergence.}
    \label{fig:ic2t3}
\end{figure}

\FloatBarrier

We observe that even with larger choice of initial data, Figures \ref{fig:ic2t1}, \ref{fig:ic2t2}, and \ref{fig:ic2t3} remain qualitatively similar to Figures \ref{fig:ic1t1}, \ref{fig:ic1t2}, and \ref{fig:ic1t3}. This computational result is again in accordance with Corollary \ref{conv_particular}.

\section{Conclusions}\label{sec_conclusions}
\noindent
We introduced new modifications of the 2D Kuramoto-Sivashinsky equation, in both scalar and vector forms, with a ``calming-parameter'' $\epsilon>0$ that we call the ``calmed Kuramoto-Sivashinsky equation,'' and proved that associated PDEs are globally well-posed in the sense of Hadamard.  Moreover, we proved that, under suitable conditions on the calming function $\bee$, that (on the time interval of existence and uniqueness of solutions to the KSE) the solutions of the calmed equation converge to solutions of the KSE as $\epsilon\rightarrow0^+$ at a certain algebraic rate.  Moreover, our computational simulations indicate that this rate is sharp.  To the best of our knowledge, this is the first globally well-posed PDE model whose solutions approximate  solutions to the 2D Kuramoto-Sivashinsky equation with arbitrary precision, at least before the potential blow-up time of the latter.

In addition, we note that this ``calming'' technique can be applied to a wide variety of other equations, which we will investigate in several forthcoming works.

\section*{Acknowledgments}
 \noindent
The authors would like to thank Professor Huy Nguyen for helpful discussions.
A.L. would like to thank the Isaac Newton Institute for Mathematical Sciences, Cambridge, for support and warm hospitality during the programme ``Mathematical aspects of turbulence: where do we stand?'' where work on this paper was undertaken. This work was supported by EPSRC grant no EP/R014604/1.
M.E. and A.L. were partially supported by NSF Grants DMS-2206762 and CMMI-1953346. 
J.W. was partially supported by NSF Grant DMS 2104682 and the AT\&T Foundation at Oklahoma State University.
\begin{scriptsize}
\bibliographystyle{abbrv}

\end{scriptsize}

\end{document}